\documentclass[article]{elsarticle}

\usepackage{lineno,hyperref}
\modulolinenumbers[5]

\usepackage{csquotes}

\usepackage[ruled,algosection]{algorithm2e}
\usepackage{geometry}
\usepackage{float}
\newcommand{\iu}{{i\mkern1mu}}
\usepackage{slashbox}
\usepackage{url}
\usepackage{pifont}
\usepackage{empheq}
\usepackage{bm}
\hypersetup{
     colorlinks   = true,
     citecolor    = blue
}
\usepackage[english]{babel} 
\addto\captionsenglish{}

\usepackage{amsmath}
\usepackage{amsmath,amssymb}
\usepackage{epstopdf}
\usepackage{relsize}
\usepackage{stmaryrd} 
\usepackage{subcaption}
\usepackage{multirow}
\usepackage{array,makecell}

\usepackage{amsthm}
\renewenvironment{proof}{{\bfseries Proof.\hspace{0.3cm}}}{\qed}

\newtheorem{thm}{Theorem}[section]
\newtheorem{lem}{Lemma}[section]
\newtheorem{cor}{Corollary}[section]
\newdefinition{rmk}{Remark}[section]
\newproof{pf}{Proof}
\newproof{pot}{Proof of Theorem \ref{thm2}}
\newdefinition{definition}{Definition}[section]

\usepackage{enumitem}

\usepackage{scalerel}
\def\stretchint#1{\vcenter{\hbox{\stretchto[220]{\displaystyle\int}{#1}}}}

\def\bs{\!\!}

\makeatletter
\def\ps@pprintTitle{%
 \let\@oddhead\@empty
 \let\@evenhead\@empty
 \def\@oddfoot{}%
 \let\@evenfoot\@oddfoot}
\makeatother
\begin{document}

\begin{frontmatter}

\title{Stability of Energy Stable Flux Reconstruction for the Diffusion Problem using the Interior Penalty and Bassi and Rebay II Numerical Fluxes}


\author[mymainaddress]{Samuel Quaegebeur
\corref{mycorrespondingauthor}}
\cortext[mycorrespondingauthor]{Corresponding author}
\ead{samuel.quaegebeur@mail.mcgill.ca}

\author[mymainaddress]{Siva Nadarajah}
\author[mymainaddress]{Farshad Navah}
\author[mymainaddress]{Philip Zwanenburg}

\address[mymainaddress]{Department of Mechanical Engineering, McGill University, Montreal, QC, H3A 0C3, Canada}

\begin{abstract}
Recovering some prominent high-order approaches such as the discontinuous Galerkin (DG) or the spectral difference (SD) methods, the flux reconstruction (FR) approach has been adopted by many individuals in the research community and is now commonly used to solve problems on unstructured grids over complex geometries. This approach relies on the use of correction functions to obtain a differential form for the discrete problem. A class of correction functions, named energy stable flux reconstruction (ESFR) functions, has been proven stable for the linear advection problem. This proof has then been extended for the diffusion equation using the local discontinuous Galerkin (LDG) scheme to compute the numerical fluxes. Although the LDG scheme is commonly used, many prefer the interior penalty (IP), as well as the Bassi and Rebay II (BR2) schemes. Similarly to the LDG proof, this article provides a stability analysis for the IP and the BR2 numerical fluxes. In fact, we obtain a theoretical condition on the penalty term to ensure stability. This result is then verified through numerical simulations. To complete this study, a von Neumann analysis is conducted to provide a combination of parameters producing the maximal time step while converging at the correct order. All things considered, this article has for purpose to provide the community with a stability condition while using the IP and the BR2 schemes.
\end{abstract}


\end{frontmatter}

\section{Introduction}
The discontinuous Galerkin method, proposed by Reed and Hill \cite{noauthor_triangular_nodate}, combines the interesting properties of finite-volume and the finite element methods. As explained in the book of Hesthaven and Warburton \cite{NDG}, it uses numerical fluxes to provide stability and high-order shape functions to represent the solution. In the last two decades, theoretical and numerical results have contributed to widen the knowledge of the method.
Concerning the diffusion equation, many numerical fluxes have been developed such as IP \cite{arnold_interior_1982}, BR2 \cite{bassi_high_2000}, LDG \cite{cockburn_local_1998} and CDG \cite{peraire_compact_2008}.
All of these methods contain a penalty term, which controls the jump between the cells or control volumes. Choosing this parameter appropriately ensures stability.

The unifying framework of flux reconstruction, developed by Huynh \cite{huynh_flux_2007}, \cite{huynh_reconstruction_2009}, enables the recovery, through the use of correction functions, of not only the DG method, but also other high-order methods such as SD~\cite{liu_spectral_2006}. These methods are of special interest since, as shown in~\cite{huynh_flux_2007}, \cite{huynh_reconstruction_2009} through a von Neumann analysis, they allow for a larger time step compared to DG; however few studies have been achieved on the stability of these methods. In order to fill this gap of knowledge, Vincent \cite{vincent_new_2011} defined a class of correction functions named Vincent Castonguay Jameson Huynh (VCJH) schemes, containing DG, SD and g2, which is stable for the linear advection problem. The mathematical proof is based on the energy of the solution and hence this class has taken the name of energy stable flux reconstruction (ESFR) schemes. This proof was extended to the diffusion equation by Castonguay et al. \cite{castonguay_energy_2013}. Employing the LDG numerical flux and the ESFR correction functions, they proved that taking a positive penalty term ensures the stability of the method. The purpose of this article is to extend the theoretical proof of~\cite{castonguay_energy_2013} to the compact IP and BR2 schemes and to provide a von Neumann analysis, calculating the maximal time steps, and the impact on the order of accuracy of the schemes.

This article is composed as follows : Section~\ref{sec:independent kappa} shows that the method is independent of the auxiliary correction function when employing the IP or the BR2 numerical fluxes. Section~\ref{sec:section 1 IP} contains the theoretical proof of stability for the IP scheme and numerical verifications, Section~\ref{sec:section 2 BR2} presents the analogy between the IP scheme and the BR2 scheme, Section~\ref{sec:Von neumann analysis} exposes a von Neumann analysis showing the dissipation and dispersion relations as well as the maximum time step. We conclude this paper with Section~\ref{sec:L2 errors} presenting the $L_{2}$-errors for a given problem.
Therefore this article allows to make a sensible choice of parameters in order to combine a stable scheme with a high time step and a correct order of accuracy.
As the theoretical result is based on the work of Castonguay et al.~\cite{castonguay_energy_2013}, we strongly advise the reader
to consult this reference \textit{a priori}. The current article will use similar notations in an attempt to be as comprehensible as possible.
\section{Preliminaries}\label{sec:Preliminaries}
We present, briefly, the flux reconstruction scheme for the diffusion problem and introduce the notations used throughout this article. The FR approach is based on the work of Huynh \cite{huynh_reconstruction_2009}.

Let us consider the diffusion equation,
\begin{equation}\label{eq:diffusion equation}
\dfrac{\partial u}{\partial t}= b\Delta u, \,\, x\in\,\Omega ,\,\,t\in\left[0,T\right]
\end{equation}
where $u\left(x,t\right)$ is the solution in the physical space, $b$ is the diffusion parameter, $\Delta$ is the Laplacian operator, $\Omega$ is the physical domain and $T$ is a positive value. We apply the tesselation $\mathcal{T}_{h}=\sum_{n=1}^{N_{K}}\Omega_{n}$ such that the spatial domain $\Omega$ is approximated by $N_{K}$ non-overlapping and continuous elements where $\Omega_{n}$ is the $n$-th element. The second-order PDE~\eqref{eq:diffusion equation} can be written as a system of two first-order equations
\begin{empheq}[left=\empheqlbrace]{align}
\dfrac{\partial u_{n}}{\partial t}&=b\nabla q_{n},\label{eq:main eq}\\
q_{n}&=\nabla u_{n},\label{eq:aux eq}
\end{empheq}
where $\nabla=\dfrac{\partial}{\partial x}$.

We refer to~\eqref{eq:main eq} as the primary equation and~\eqref{eq:aux eq} as the auxiliary equation.
Equations~\eqref{eq:main eq} and~\eqref{eq:aux eq} are valid on each element $\Omega_{n}=\left[x_{n},x_{n+1}\right]$, $n\in \llbracket 1,N_{K} \rrbracket$ (interval of the type $\llbracket .,.\rrbracket$ denotes an integer interval). We note $u_{n}$ and $q_{n}$ the solution and the corrected gradient on the element $\Omega_{n}$. We then map the equation from the physical into the computational space through the affine mapping $\mathcal{M}_{n}$
\begin{equation}\label{eq:affine mapping}
\begin{array}{lccl}
\mathcal{M}_{n}\colon &\left[-1 , 1\right]&\to& \Omega_{n}\\
&r&\mapsto&\dfrac{\left(1-r\right)}{2}x_{n}+\dfrac{\left(1+r\right)}{2}x_{n+1}.
\end{array}
\end{equation}
The Jacobian of this mapping in element $\Omega_{n}$ is noted as $J_{n}$,
\begin{equation}\label{eq:def jacobian}
J_{n}=\dfrac{\mbox{d}x}{\mbox{d}r}=\dfrac{\mbox{d} \mathcal{M}_{n}\left(r\right)}{\mbox{d}r}=\dfrac{x_{n+1}-x_{n}}{2}=\dfrac{h_{n}}{2},
\end{equation}
where $h_{n}$ is the length of element $\Omega_{n}$.
Using this mapping, we write equations~\eqref{eq:main eq} and~\eqref{eq:aux eq} in the computational domain as
\begin{empheq}[left=\empheqlbrace]{align}
\dfrac{\partial u_{n}}{\partial t}&=b\dfrac{1}{J_{n}}\nabla_{r} q_{n},\label{eq:main eq compu}\\
q_{n}&=\dfrac{1}{J_{n}}\nabla_{r} u_{n},\label{eq:aux eq compu}
\end{empheq}
where $\nabla_{r}=\dfrac{\partial}{\partial r}$.

Equations~\eqref{eq:main eq compu} and~\eqref{eq:aux eq compu} are insufficient to describe the phenomenom occuring in the entire domain. To take into account the transfer of information between the cells, we must add a numerical flux to each equation. In this article, we will focus on the IP and BR2 schemes. We note $\left(u\right)^{*}$ as the numerical flux associated to the auxiliary equation and $\left(q\right)^{*}$ as the numerical flux associated to the primary equation. We then express equations \eqref{eq:main eq compu} and \eqref{eq:aux eq compu} as

\begin{empheq}[left=\empheqlbrace]{align}
\dfrac{\partial u_{n}}{\partial t}&=b\dfrac{1}{J_{n}}\nabla_{r}\left[q+\left.\left(q_{n}^{*}-q_{n}\right)\right|_{-1}h_{L}+\left.\left(q_{n}^{*}-q_{n}\right)\right|_{1}h_{R}\right],\label{eq:main eq compu tot}\\
q_{n}&=\dfrac{1}{J_{n}}\nabla_{r}\left[u_{n}+\left.\left(u^{*}-u_{n}\right)\right|_{-1}g_{L}+\left.\left(u^{*}-u_{n}\right)\right|_{1}g_{R}\right],\label{eq:aux eq compu tot}
\end{empheq}
where $h_{L}$ and $h_{R}$ are the FR correction functions, defined in the computational space, for respectively the left and right boundaries of the primary equation while $g_{L}$ and $g_{R}$ are the corresponding correction functions for the auxiliary equation.

The main idea of these correction functions is to create continuous quantities through the edges of the element. Indeed, $u$ is usually discontinuous between elements $\Omega_{n}$ and $\Omega_{n+1}$. By adding the correction $\left.\left(u^{*}-u_{n}\right)\right|_{1}g_{R}$ to the solution $u_{n}$ on the element $\Omega_{n}$ and the correction $\left.\left(u^{*}-u_{n+1}\right)\right|_{-1}g_{L}$ to the solution $u_{n+1}$ on the element $\Omega_{n+1}$, we obtain a continuous quantity $u$ on $\Omega_{n}\cup \Omega_{n+1}$. A similar procedure is done for the quantity $q$ in the primary equation. To create these continuous quantities, the correction functions $g_{L}$ (respectively $g_{R}$) and $h_{L}$ (respectively $h_{R}$) must respect particular features \cite{huynh_reconstruction_2009},
\begin{empheq}[left=\empheqlbrace]{align}
g_{L}\left(-1\right)=g_{R}\left(1\right)=1,\label{eq:properties correction functions 1}\\
g_{R}\left(-1\right)=g_{L}\left(1\right)=0,\label{eq:properties correction functions 2}
\end{empheq}
and $h_{L}$ and $h_{R}$ respect the same properties. Moreover, the correction functions $g_{L}$ (respectively $g_{R}$) and $h_{L}$ (respectively $h_{R}$) are polynomials of degree $p+1$. Indeed, as we will see shortly the solution at the time $t_{n}$, $u_{n}^{t_{n}}$, is a polynomial of degree $p$ and therefore, for $u_{n}^{t_{n}+1}$ to be also a polynomial of degree $p$, we require the correction functions to be polynomials of degree $p+1$.

However, just respecting the above stated properties does not ensure the stability of the scheme. For the advection equation, Vincent et al. defined a class (VCJH) of correction functions \cite{vincent_new_2011}; if the correction function $g_{L}$ (similarly $g_{R}$) satisfies the following property (see Eq.~(3.32) in~\cite{vincent_new_2011}) then the scheme is energy stable for the advection equation, using the Lax-Friedrichs numerical flux
\begin{equation}\label{eq:ESFR correction functions}
\int_{-1}^{1}g_{L}\dfrac{\partial l_{i}}{\partial r} \mbox{d}r=c\left(\dfrac{\partial^{p}l_{i}}{\partial r^{p}}\right)\left(\dfrac{\partial^{p+1}g_{L}}{\partial r^{p+1}}\right),
\end{equation}
where $c$ is a scalar and $l_{i}$ is a Lagrange polynomial, which will be defined shortly. For $c\in\left[c_{-},\infty\right[$, the corresponding correction function $g$ provides a stable scheme for the advection equation.
 $c_{-}$ is a negative value, given in \cite{vincent_new_2011}.

The goal of this article is to obtain a stability condition for the advection-diffusion problem. Choosing the ESFR correction function only ensures the stability for the advection part. A new condition on the correction function for the diffusion term must be developed to ensure a stable scheme. 

The diffusion equation combines both the auxiliary and primary equations and thus requires two correction functions $g_{L}$ (respectively $g_{R}$) and $h_{L}$ (respectively $h_{R}$), where both ESFR correction functions are parametrized by a different scalar: $c$ for the $h$ correction function and $\kappa$ for the $g$ correction function.

Numerically, we do not require the analytical formulation of $g_{L}$ (respectively $g_{R}$) and $h_{L}$ (respectively $h_{R}$) but of their derivatives. The final expression for the one-dimensional diffusion equation reads,
\begin{empheq}[left=\empheqlbrace]{align}
\dfrac{\partial u_{n}}{\partial t}&=b\dfrac{1}{J_{n}}\left[\nabla_{r} q_{n}+\left.\left(q^{*}-q_{n}\right)\right|_{-1}\nabla_{r}h_{L}+\left.\left(q^{*}-q_{n}\right)\right|_{1}\nabla_{r}h_{R}\right],\label{eq:numerical main}\\
q_{n}&=\dfrac{1}{J_{n}}\left[\nabla_{r} u_{n}+\left.\left(u^{*}-u_{n}\right)\right|_{-1}\nabla_{r}g_{L}+\left.\left(u^{*}-u_{n}\right)\right|_{1}\nabla_{r}g_{R}\right],\label{eq:numerical aux}
\end{empheq}
where $\nabla_{r}g$ and $\nabla_{r}h$ are the derivatives in the computational space.

To complete the formulation, it remains to describe the numerical fluxes for the choosen schemes. We will present the LDG, the IP, and the BR2 numerical fluxes~\cite{NDG}.
\begin{equation}\label{eq:numerical fluxes}
\left\lbrace
\begin{aligned}
&\text{LDG:}\\
&u^{*}=\left\lbrace\left\lbrace u\right\rbrace\right\rbrace -\beta\llbracket u\rrbracket \\
&q^{*}=\left\lbrace\left\lbrace q\right\rbrace\right\rbrace +\beta \llbracket q\rrbracket-\tau \llbracket u\rrbracket
\end{aligned}
\right.
\left\lbrace
\begin{aligned}
&\text{IP:}\\
&u^{*}=\left\lbrace\left\lbrace u\right\rbrace\right\rbrace  \\
&q^{*}=\left\lbrace\left\lbrace\nabla u\right\rbrace\right\rbrace -\tau \llbracket u\rrbracket
\end{aligned}
\right.
\left\lbrace
\begin{aligned}
&\text{BR2:}\\
&u^{*}=\left\lbrace\left\lbrace u\right\rbrace\right\rbrace  \\
&q^{*}=\left\lbrace\left\lbrace\nabla u\right\rbrace\right\rbrace +s  \left\lbrace\left\lbrace r^{e}\left(\llbracket u\rrbracket\right)\right\rbrace\right\rbrace
\end{aligned}
\right.
\end{equation}
In the LDG scheme, $\beta$ is a scalar value, in 1D, while $\tau$ signifies the penalty term
for both LDG and IP, $s$ denotes the penalty term for the BR2 scheme. Lastly, $r^{e}$ is a lifting operator and is defined as
\begin{equation}
\displaystyle{
\int_{\Omega}r^{e}\left(\llbracket u\rrbracket\right)\phi\,\mbox{d}x=-\int_{e}\llbracket u\rrbracket \left\lbrace\left\lbrace \phi\right\rbrace\right\rbrace \,\mbox{d}s,
}
\end{equation}
which can be written in 1D as

\begin{equation}\label{eq:def re}
\displaystyle{
\int_{\Omega}r^{e}\left(\llbracket u\rrbracket\right)\phi\,\mbox{d}x=-\llbracket u_{e}\rrbracket\left\lbrace\left\lbrace \phi_{e}\right\rbrace\right\rbrace,
}
\end{equation}
where $\phi$ is a test function. The symbol $\llbracket\,\,\rrbracket$ denotes a jump (it must not be confused with the integer interval $\llbracket .,.\rrbracket$) and $\{\{\,\,\}\}$ a mean value. The former notation defines the  discontinuity of the solution across an edge, while the latter signifies the average of the solution across an edge.
\begin{figure}[H]
\centering
\includegraphics[width=10cm]{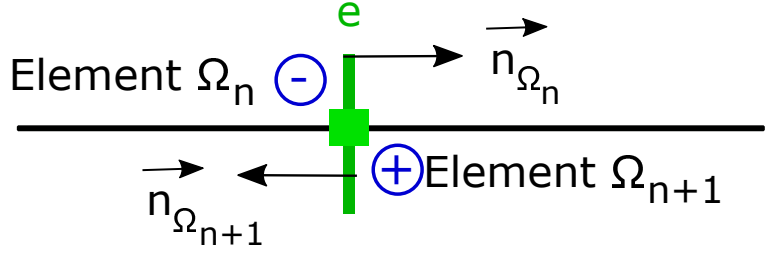}
\caption{Normal vector convention on an edge.}
\label{fig:jump}
\end{figure}

Typically, the quantity $u$ on the left side of an edge $e$ is referred as $u_{e,-}$ and on the right side of the edge: $u_{e,+}$. Referring to Figure~\ref{fig:jump}, we can define the jump and the mean value as,

\begin{eqnarray}
\llbracket u_{e} \rrbracket&=&u_{e,-}-u_{e,+},\label{eq:jump}\vspace{0.1cm}\\
\left\lbrace\left\lbrace u_{e}\right\rbrace\right\rbrace &=&\dfrac{1}{2}\left(u_{e,-}+u_{e,+}\right).\label{eq:mean}
\end{eqnarray}

\begin{figure}[H]
\centering
\includegraphics[width=15cm]{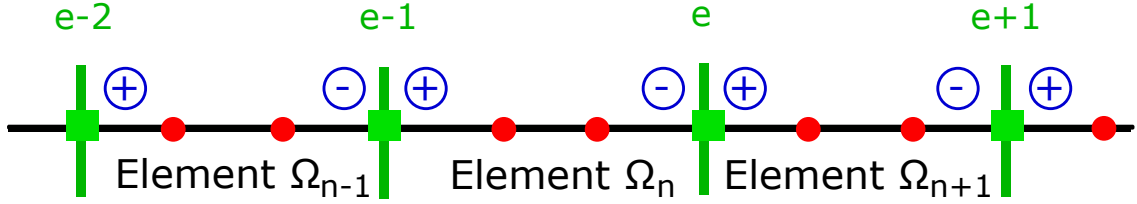}
\caption{Representation of several elements for $p=1$.}
\label{fig:edges_elements}
\end{figure}

Figure~\ref{fig:edges_elements} describes the interface $e$ located between elements $\Omega_{n}$ and $\Omega_{n+1}$. For instance $u_{e,-}$ represents the value taken by the solution on the element $\Omega_{n}$ on its right edge (green square on edge $e$).
Whereas  $u_{e,+}$ represents the value taken by the solution on the Element $\Omega_{n+1}$ on its left edge (green square on edge $e$). The red circles represent the $p+1$ solution points within any standard element, while the green square denotes the flux points. Indeed the solution on an element $\Omega_{n}$ will be calculated at $N_{p}$ nodes within the element. Using Lagrange polynomials we can interpolate the solution with a polynomial of degree $p=N_{p}-1$. The nodes, $r_{i}$, are typically chosen as the Lobatto-Gauss-Legendre (LGL) or the Gauss-Legendre (GL) nodes.
Finally the solution $u_{n}$ on the element $\Omega_{n}$ can be written as,
\begin{equation}\label{solution interpolated}
\displaystyle{
u_{n}=\sum_{i=1}^{N_{p}}\hat{u}_{i}l_{i},
}
\end{equation}
where $\hat{u}_{i}$ is the value of $u$ at the nodes $r_{i}$ and $l_{i}$ is the associated Lagrange polynomial,
\begin{equation}\label{eq:lagrange}
\displaystyle
l_{i}\left(r\right)=\prod_{\substack{j=1\\j\neq i}}^{N_{P}}\left(\dfrac{r-r_{j}}{r_{i}-r_{j}}\right).
\end{equation}

The number of nodes $N_{p}$ and hence the degree of the polynomial $p$ is of the utmost importance as with high-order methods we should get an error of order $\mathcal{O}\left(h^{p+1}\right)$ in the $L_{2}$ norm, where $h$ is a typical element size. We have
\begin{equation}
\Vert u \Vert_{0,\Omega}=\left[\int_{\Omega}u^{2}\,\mbox{d}x\right]^{1/2},
\end{equation}
the $L_{2}$ norm of $u$.

In summary, we can enumerate the different choices available for the FR scheme as,
\begin{enumerate}
\item the choice of the correction functions $g_{L}$ and $g_{R}$ is selected through the parameter $\kappa$. As we will show in the following section, the method is independent of $\kappa$ when employing the IP or the BR2 numerical fluxes. This theoretical result will be verified for two values of $\kappa$: $\kappa_{DG}$ and $\kappa_{+}$, which are particular values explained below. This article also contains the results of the LDG numerical flux. For this scheme, $\kappa$ will have a wide range of values: $\kappa\in[0,10^{8}]$.
\item the choice of the correction functions $h_{L}$ and $h_{R}$ is selected through the parameter $c$. These correction functions pertain to the advective term, if there is one. Many studies have been conducted for the pure advective equation. In the dissertation of Castonguay \cite{castonguay2012high}, it was shown that for too large a value of $c$, we lose one order of accuracy. A particular value of $c$ noted $c_{+}$ was defined as the parameter providing the maximal time step $\Delta t_{max}$ while preserving the order of accuracy (OOA). Therefore, in our analysis, we will mainly focus on a range of $c$ in: $\left[0, c_{+}\right]$.
In this paper we will present three particular schemes: DG (discontinuous Galerkin), SD (spectral difference), and HU (g2 method). For brevity, their analytical expressions will not be given, only the numerical values. One can find them in the paper of Vincent et al.~\cite{vincent_insights_2011}. 

\begin{table}[H]
\centering
\makebox[\textwidth][c]{
\begin{tabular}{|c|c|c|c|c|}
\hline
\bf \backslashbox{$c$}{$p$}&2&3&4&5\\
\hline
$c_{DG}$&0&0&0&0\\
\hline
$c_{SD}$&2.96e-02&9.52e-04&1.61e-05&1.70e-07\\
\hline
$c_{HU}$&6.67e-02&1.69e-03&2.52e-05&2.44e-07\\
\hline
$c_{+}$&1.86e-01&3.67e-03&4.79e-05&4.24e-07\\
\hline
\end{tabular}
}
\caption{Numerical values of the correction function parameter for four particular methods.}
\label{tab:values of c}
\end{table}
Although the three first rows of Table~\ref{tab:values of c} are obtained by analytical formulas, the fourth row, corresponding to $c_{+}$ are obtained via a von Neumann analysis for the RK54 scheme \cite{carpenter_fourth-order_1994}. These values are a bit different from the one obtained by Vincent et al.~\cite{vincent_insights_2011}. We suspect numerical approximations to be the source of these discrepancies.

\item the choice of number and position of nodes. Although the simulations have been done for equidistant, GL and LGL points; only the results of LGL will be presented for brevity, since the trends are similar. The degree of the polynomial has been taken to be $\llbracket 2,5\rrbracket$, but only $p=2$ and $p=3$ will be presented. Results for all other $p$ demonstrated similar results. 
\item the penalty term $\tau$ for the IP scheme and $s$ for the BR2 scheme.
\end{enumerate} 

\section{The diffusion equation, independent of $\kappa$}\label{sec:independent kappa}
The purpose of this section is to prove that using the ESFR schemes with the IP and the BR2 numerical fluxes results in an independency of the problem from $\kappa$.

The IP and the BR2 numerical fluxes described in equation~\eqref{eq:numerical fluxes} yield for an element $\Omega_{n}$,
\begin{eqnarray}
\left.\left(u^{*}-u_{n}\right)\right|_{-1}=\dfrac{\llbracket u_{e-1}\rrbracket}{2},\vspace{0.2cm}\\
\left.\left(u^{*}-u_{n}\right)\right|_{1}=-\dfrac{\llbracket u_{e}\rrbracket}{2}.
\end{eqnarray}

The IP and BR2 schemes differ only by $q^{*}$. However for both schemes, there is no dependence on $q$ and hence on $\kappa$. For more generality, we will not replace $q^{*}$ by its expression in the following derivations.

We introduce the previous equalities in equations~\eqref{eq:numerical main} and~\eqref{eq:numerical aux},
\begin{empheq}[left=\empheqlbrace]{align}
\dfrac{\partial u_{n}}{\partial t}&=b\dfrac{1}{J_{n}}\left[\nabla_{r} q_{n}+\left.\left(q^{*}-q_{n}\right)\right|_{-1}\nabla_{r}h_{L}+\left.\left(q^{*}-q_{n}\right)\right|_{1}\nabla_{r}h_{R}\right],\label{eq:numerical main bis}\\
q_{n}&=\dfrac{1}{J_{n}}\left[\nabla_{r} u_{n}+\dfrac{\llbracket u_{e-1}\rrbracket}{2}\nabla_{r}g_{L}-\dfrac{\llbracket u_{e}\rrbracket}{2}\nabla_{r}g_{R}\right].\label{eq:numerical aux bis}
\end{empheq}

We now substitute equation~\eqref{eq:numerical aux bis} into equation~\eqref{eq:numerical main bis},
\begin{equation}\label{eq:ini demo}
\begin{array}{lll}
\dfrac{\partial u_{n}}{\partial t}&=&\dfrac{b}{J_{n}}\left[\dfrac{1}{J_{n}}\Delta_{r} u_{n}+\left.\left(q^{*}-\dfrac{1}{J_{n}}\nabla_{r}u_{n}\right)\right|_{-1}\nabla_{r}h_{L}+\left.\left(q^{*}-\dfrac{1}{J_{n}}\nabla_{r}u_{n}\right)\right|_{1}\nabla_{r}h_{R}\right]\vspace{0.2cm}\\
&&+\dfrac{b\llbracket u_{e-1}\rrbracket}{2J_{n}^{2}}\left[\Delta_{r}g_{L}-\nabla_{r}g_{L}\left(-1\right)\nabla_{r}h_{L}-\nabla_{r}g_{L}\left(1\right)\nabla_{r}h_{R}\right]\vspace{0.2cm}\\
&&-\dfrac{b\llbracket u_{e}\rrbracket}{2J_{n}^{2}}\left[\Delta_{r}g_{R}-\nabla_{r}g_{R}\left(-1\right)\nabla_{r}h_{L}-\nabla_{r}g_{R}\left(1\right)\nabla_{r}h_{R}\right],
\end{array}
\end{equation}
where $\Delta_{r}=\dfrac{\partial^{2}}{\partial r^{2}}$.
The first line of equation~\eqref{eq:ini demo} does not depend on $\kappa$; but the last two contain the correction function, $g$ and hence could potentially depend on $\kappa$.

\begin{thm}\label{thm:part 1 independent}
Employing the ESFR correction functions defined in Definition~\ref{def:Left} and~\ref{def:right}, the following equation is independent of $\kappa$,
\begin{equation}
\Delta_{r}g_{L}-\nabla_{r}g_{L}\left(-1\right)\nabla_{r}h_{L}-\nabla_{r}g_{L}\left(1\right)\nabla_{r}h_{R},
\end{equation}
where $g_{L}$ and $g_{R}$ are parametrized by $\kappa$ and $h_{L}$ and $h_{R}$ by $c$.
\end{thm}
\begin{proof}
In the following derivations, the authors have used the properties of the Legendre polynomials, which are recalled in section~\ref{sec:Properties Legendre polynomial}. From the Definition~\ref{def:Left} of the left-ESFR correction function and from~\ref{prop:derivative value +1 -1}, we have,
\begin{equation}
\begin{array}{cll}
\nabla_{r}g_{L}\left(-1\right)&=&\dfrac{\left(-1\right)^{p}}{2}\left[\dfrac{\left(-1\right)^{p+1}p\left(p+1\right)}{2}-\dfrac{\eta_{p,\kappa}\left(-1\right)^{p}p\left(p-1\right)+\left(-1\right)^{p}\left(p+1\right)\left(p+2\right)}{2\left(1+\eta_{p,\kappa}\right)}\right]\\
&=&-\dfrac{1}{4}\left[p\left(p+1\right)+\dfrac{\eta_{p,\kappa}p\left(p-1\right)+\left(p+1\right)\left(p+2\right)}{\left(1+\eta_{p,\kappa}\right)}\right].
\end{array}
\end{equation}
Similarly, we obtain,
\begin{equation*}
\nabla_{r}g_{L}\left(1\right)=\dfrac{\left(-1\right)^{p}}{4}\left[p\left(p+1\right)-\dfrac{\eta_{p,\kappa}p\left(p-1\right)+\left(p+1\right)\left(p+2\right)}{\left(1+\eta_{p,\kappa}\right)}\right].
\end{equation*}

We now multiply the previous equations by $\nabla_{r}h_{L}$ (respectively $\nabla_{r}h_{R}$),
\begin{equation*}
\begin{array}{lll}
\nabla_{r}g_{L}\left(-1\right)\nabla_{r}h_{L}&=&\dfrac{\left(-1\right)^{p}}{8}\left[p\left(p+1\right)+\dfrac{\eta_{p,\kappa}p\left(p-1\right)+\left(p+1\right)\left(p+2\right)}{\left(1+\eta_{p,\kappa}\right)}\right]\left(-\Psi_{p}^{'}+\dfrac{\eta_{p,c}\Psi_{p-1}^{'}+\Psi_{p+1}^{'}}{\left(1+\eta_{p,c}\right)}\right),\\
\nabla_{r}g_{L}\left(1\right)\nabla_{r}h_{R}&=&\dfrac{\left(-1\right)^{p}}{8}\left[p\left(p+1\right)-\dfrac{\eta_{p,\kappa}p\left(p-1\right)+\left(p+1\right)\left(p+2\right)}{\left(1+\eta_{p,\kappa}\right)}\right]\left(\Psi_{p}^{'}+\dfrac{\eta_{p,c}\Psi_{p-1}^{'}+\Psi_{p+1}^{'}}{\left(1+\eta_{p,c}\right)}\right),
\end{array}
\end{equation*}
where $\Psi_{p}^{'}$ denotes the derivative of the Legendre polynomial of degree $p$.

Then, we add the above two equations to yield,
\begin{equation}\label{eq:sum gl hr}
\begin{array}{lll}
\nabla_{r}g_{L}\left(-1\right)\nabla_{r}h_{L}+\nabla_{r}g_{L}\left(1\right)\nabla_{r}h_{R}&=&\dfrac{\left(-1\right)^{p}}{4}p\left(p+1\right)\left(\dfrac{\eta_{p,c}\Psi_{p-1}^{'}+\Psi_{p+1}^{'}}{\left(1+\eta_{p,c}\right)}\right)\\
&&+\dfrac{\left(-1\right)^{p+1}}{4}\left(\dfrac{\eta_{p,\kappa}p\left(p-1\right)+\left(p+1\right)\left(p+2\right)}{\left(1+\eta_{p,\kappa}\right)}\Psi_{p}^{'}\right).
\end{array}
\end{equation}

Next, we shift our focus to the Laplacian of the correction function, differentiating~\ref{prop:derivative} with respect to $r$ yields,
\begin{equation}\label{eq:delta gl}
\begin{array}{lll}
\Delta_{r} g_{L}&=&\dfrac{\left(-1\right)^{p}}{2}\left[\Psi_{p}^{''}-\dfrac{\eta_{p,\kappa}\Psi_{p-1}^{''}+\Psi_{p+1}^{''}}{\left(1+\eta_{p,\kappa}\right)}\right]\vspace{0.2cm}\\
&=&\dfrac{\left(-1\right)^{p}}{2}\left[\Psi_{p}^{''}-\Psi_{p-1}^{''}-\dfrac{\left(2p+1\right)\Psi_{p}^{'}}{\left(1+\eta_{p,\kappa}\right)}\right].
\end{array}
\end{equation}

Substracting equation~\eqref{eq:sum gl hr} from~\eqref{eq:delta gl},
\begin{equation}
\begin{array}{lll}
\Delta_{r}g_{L}-\left(\nabla_{r}g_{L}\left(-1\right)\nabla_{r}h_{L}+\nabla_{r}g_{L}\left(1\right)\nabla_{r}h_{R}\right)&=&\dfrac{\left(-1\right)^{p}}{2}\left[\Psi_{p}^{''}-\Psi_{p-1}^{''}\right]\\
&+&\dfrac{\left(-1\right)^{p+1}}{4}p\left(p+1\right)\left(\dfrac{\eta_{p,c}\Psi_{p-1}^{'}+\Psi_{p+1}^{'}}{\left(1+\eta_{p,c}\right)}\right)\\
&+&\dfrac{\left(-1\right)^{p}}{4}\left(\dfrac{\eta_{p,\kappa}p\left(p-1\right)+\left(p+1\right)\left(p+2\right)-2\left(2p+1\right)}{\left(1+\eta_{p,\kappa}\right)}\Psi_{p}^{'}\right)\vspace{0.2cm}\\
&=&\dfrac{\left(-1\right)^{p}}{2}\left[\Psi_{p}^{''}-\Psi_{p-1}^{''}\right]\\
&+&\dfrac{\left(-1\right)^{p+1}}{4}p\left(p+1\right)\left(\dfrac{\eta_{p,c}\Psi_{p-1}^{'}+\Psi_{p+1}^{'}}{\left(1+\eta_{p,c}\right)}\right)\\
&+&\dfrac{\left(-1\right)^{p}}{4}p\left(p-1\right)\Psi_{p}^{'},
\end{array}
\end{equation}
where the final form only contains $\eta_{p,c}$.
\end{proof}

\begin{rmk}\label{rmk:extension gr}
The same property can be said for $\Delta_{r}g_{R}-\nabla_{r}g_{R}\left(-1\right)\nabla_{r}h_{L}-\nabla_{r}g_{R}\left(1\right)\nabla_{r}h_{R}$, we need only to adapt the previous proof to $g_{R}$.
\end{rmk}

\begin{cor}\label{cor:independent}
The diffusion equation is independent of $\kappa$ when employing the IP or the BR2 numerical fluxes.
\end{cor}
\begin{proof}
The diffusion equation with the ESFR schemes and employing the IP or the BR2 numerical fluxes is written as equation~\eqref{eq:ini demo}. Only the last two lines of this PDE contain the parameter $\kappa$. However from Theorem~\ref{thm:part 1 independent} and the remark~\ref{rmk:extension gr} the parameter $\kappa$ gets simplified and as a consequence equation~\eqref{eq:ini demo} does not depend on $\kappa$.
\end{proof}
\section{IP Stability Condition}\label{sec:section 1 IP}
\subsection{Theoretical results}

Having introduced the different notations and definitions, we can derive the stability condition for the IP numerical flux using the ESFR class functions. We will, briefly, explain the main result of~\cite{castonguay_energy_2013} and explain how, from this analysis, we can extend the proof. The main difference between the notations of~\cite{castonguay_energy_2013} and ours is that we have dropped the superscript $-^{D}$ which indicated \enquote{discontinuous} quantities. Castonguay et al. derived a stability condition for the LDG numerical flux \cite{castonguay_energy_2013}
by introducing the norms
\begin{eqnarray}\label{eq:norm ESFR}
\Vert U \Vert _{p,c}&=&
\displaystyle{
\left\lbrace\sum_{n=1}^{N_{K}}\int_{x_{n}}^{x_{n+1}}\left[\left(u_{n}\right)^{2}+\dfrac{c}{2}\left(J_{n}\right)^{2p}\left(\dfrac{\partial^{p}u_{n}}{\partial x^{p}}\right)^{2}\right]\mbox{d}x\right\rbrace^{1/2}
}
\end{eqnarray}
and
\begin{eqnarray}
\Vert Q\Vert_{p,\kappa}&=&\displaystyle{
\left\lbrace\sum_{n=1}^{N_{K}}\int_{x_{n}}^{x_{n+1}}\left[\left(q_{n}\right)^{2}+\dfrac{\kappa}{2}\left(J_{n}\right)^{2p}\left(\dfrac{\partial^{p}q_{n}}{\partial x^{p}}\right)^{2}\right]\mbox{d}x\right\rbrace^{1/2}
}.
\end{eqnarray}

These norms are used to obtain particular properties of the correction functions $g$ and $h$ and therefore simplify drastically the equations of stability. After applying the simplifications, they showed for the LDG case (Equation (77), \cite{castonguay_energy_2013}) that
\begin{equation}\label{eq:Result LDG}
\displaystyle{
\dfrac{1}{2}\dfrac{\mbox{d}}{\mbox{d}t}\Vert U \Vert^{2}_{p,c}=-b\Vert Q \Vert^{2}_{p,\kappa}-\sum_{e=1}^{N_{e}}\left[\dfrac{\lambda}{2}|a|\left(u_{e,+}-u_{e,-}\right)^{2}+\tau\left(u_{e,+}-u_{e,-}\right)^{2}\right]
},
\end{equation}
where, $N_{e}$ is the number of edges, $\lambda$ is the parameter of the Lax-Friedrichs numerical flux and $a$ is the velocity of the advection.

The conclusion of their stability proof is that for all $\lambda\geq 0$ and $\tau \geq 0$,  
\begin{equation}
\displaystyle{
\dfrac{1}{2}\dfrac{\mbox{d}}{\mbox{d}t}\Vert U \Vert^{2}_{p,c}\leq 0
}.
\end{equation}

From this final result~\eqref{eq:Result LDG}, it is important to notice that the advective part, represented by $\frac{\lambda}{2}|a|\left(u_{e,+}-u_{e,-}\right)^{2}$, is independent from the diffusive part $\tau\left(u_{e,+}-u_{e,-}\right)^{2}$. Therefore, by taking any other diffusive scheme, the advective part will not change. In order, for our proof, to be as concise as possible, we will not consider the advective part (as it will not differ from the article of Castonguay et al.). Due to the similarities with~\cite{castonguay_energy_2013}, the authors have chosen not to re-derive every step of the derivations. We will start from a particular equation of the aforementioned article. To have a complete understanding of the proof, the authors strongly recommend to read article \cite{castonguay_energy_2013} (up to Equation (69)), which has been derived for any diffusive numerical flux.

Upon removing the advective term of (Equation (69) of \cite{castonguay_energy_2013}) we have,
\begin{equation}\label{eq:Equation69}
\displaystyle{
\dfrac{1}{2}\dfrac{\mbox{d}}{\mbox{d}t}\Vert U\Vert^{2}_{p,c}=-b\Vert Q\Vert_{p,\kappa}^{2}+b\sum_{e=1}^{N_{e}}\Theta_{e}
},
\end{equation}
where $\Theta_{e}=\left[
\left(u_{e,+}q_{e,+}-u_{e,-}q_{e,-}\right)
-\left(q\right)^{*}_{e}\left(u_{e,+}-u_{e,-}\right)
-\left(u\right)^{*}_{e}\left(q_{e,+}-q_{e,-}\right)
\right]$. Since the diffusive parameter $b$ is a positive constant and multiplies every term, we have factored it out from $\Theta_{e}$. This is a slight difference from the article~\cite{castonguay_energy_2013} in Eq. (75) where $b$ does not multiply the penalty term. The choice made does not impact the result of this article.

\begin{lem}
Employing the ESFR correction functions and the IP scheme, the following equality holds\begin{equation}\label{eq:last exact form}
\Theta_{e}=-\llbracket u_{e} \rrbracket^{2}\left[\tau+\dfrac{1}{4}\nabla_{r}g_{L}\left(-1\right)\left(\dfrac{1}{J_{n}}+\dfrac{1}{J_{n+1}}\right)\right]-\dfrac{1}{4}\nabla_{r}g_{L}\left(1\right)\llbracket u_{e} \rrbracket\left(\dfrac{1}{J_{n}}\llbracket u_{e-1} \rrbracket+\dfrac{1}{J_{n+1}}\llbracket u_{e+1} \rrbracket\right).
\end{equation}
\end{lem}

\begin{proof}
We expand the expression for $\Theta_{e}$ by employing the IP numerical fluxes defined in~\eqref{eq:numerical fluxes},\begin{equation}
\Theta_{e}=\left(u_{e,+}q_{e,+}-u_{e,-}q_{e,-}\right)
-\left(\left\lbrace\left\lbrace \nabla u_{e}\right\rbrace\right\rbrace -\tau\llbracket u_{e} \rrbracket\right)\left(u_{e,+}-u_{e,-}\right)
-\left(\left\lbrace\left\lbrace u_{e}\right\rbrace\right\rbrace \right)\left(q_{e,+}-q_{e,-}\right).
\end{equation}
By expanding the mean value, defined in~\eqref{eq:mean} and using the definition of the jump from~\eqref{eq:jump}, we obtain
\begin{equation}
\begin{array}{lll}
\Theta_{e}&=&\left(u_{e,+}q_{e,+}-u_{e,-}q_{e,-}\right)
+\left(\left\lbrace\left\lbrace \nabla u_{e}\right\rbrace\right\rbrace -\tau\llbracket u_{e} \rrbracket\right)\llbracket u_{e} \rrbracket
-\frac{1}{2}\left(u_{e,+}+u_{e,-}\right)\left(q_{e,+}-q_{e,-}\right)\\
&=& u_{e,+} q_{e,+}\left(1-\frac{1}{2}\right)\\
&+&u_{e,+} q_{e,-}\left(\frac{1}{2}\right)\\
&+& u_{e,-} q_{e,-}\left(-1+\frac{1}{2}\right)\\
&+& u_{e,-} q_{e,+}\left(-\frac{1}{2}\right)\\
&+&\left(\left\lbrace\left\lbrace \nabla u_{e}\right\rbrace\right\rbrace -\tau\llbracket u_{e} \rrbracket\right)\llbracket u_{e} \rrbracket.
\end{array}
\end{equation}
Using, once more the definition of the mean value and the jump, we reduce the expression to
\begin{equation}
\begin{array}{lll}
\Theta_{e}&=& u_{e,+}\left\lbrace\left\lbrace q_{e}\right\rbrace\right\rbrace -u_{e,-}\left\lbrace\left\lbrace q_{e}\right\rbrace\right\rbrace +\left\lbrace\left\lbrace \nabla u_{e}\right\rbrace\right\rbrace \llbracket u_{e} \rrbracket-\tau\llbracket u_{e} \rrbracket^{2}\\
&=&-\llbracket u_{e} \rrbracket \left\lbrace\left\lbrace  q_{e} -\nabla u_{e} \right\rbrace\right\rbrace  -\tau \llbracket u_{e} \rrbracket^{2}.\label{eq:jump q nabla 2}
\end{array}
\end{equation}

In order to simplify equation~\eqref{eq:jump q nabla 2}, we write the auxiliary variable (equation~\eqref{eq:numerical aux}), for an element $\Omega_{n}$, in terms of $u$,
\begin{equation}
q_{n}=\nabla u_{n} +\dfrac{1}{J_{n}}\left(u^{*}_{e-1}-u_{e-1,+}\right) \nabla_{r}g_{L} +\dfrac{1}{J_{n}}\left(u^{*}_{e}-u_{e,-}\right) \nabla_{r}g_{R},
\end{equation}
where $q$ and $\nabla u$ are in the physical space, while $\nabla_{r}g_{L}$ and $\nabla_{r}g_{R}$ are in the computational space. Therefore, referring to Figure~\ref{fig:edges_elements}, we have
\begin{empheq}[left=\empheqlbrace]{align}
q_{e,+}&=\nabla u_{e,+}+\dfrac{1}{J_{n+1}}\left(u^{*}_{e}-u_{e,+}\right) \nabla_{r}g_{L}\left(-1\right) +\dfrac{1}{J_{n+
1}}\left(u^{*}_{e+1}-u_{e+1,-}\right) \nabla_{r}g_{R}\left(-1\right), \vspace{0.2cm} \\
q_{e,-}&=\nabla u_{e,-}+\dfrac{1}{J_{n}}\left(u^{*}_{e-1}-u_{e-1,+}\right) \nabla_{r}g_{L}\left(1\right) +\dfrac{1}{J_{n}}\left(u^{*}_{e}-u_{e,-}\right) \nabla_{r}g_{R}\left(1\right).
\end{empheq}
Using the definition of the numerical flux of the IP scheme~\eqref{eq:numerical fluxes} and using the definition of the jump~\eqref{eq:jump}, we obtain

\begin{empheq}[left=\empheqlbrace]{align}
q_{e,+}&=\nabla u_{e,+}+\dfrac{1}{2J_{n+
1}}\llbracket u_{e} \rrbracket \nabla_{r}g_{L}\left(-1\right)-\dfrac{1}{2J_{n+
1}}\llbracket u_{e+1} \rrbracket \nabla_{r}g_{R}\left(-1\right), \vspace{0.2cm} \\
q_{e,-}&=\nabla u_{e,-}+\dfrac{1}{2J_{n}}\llbracket u_{e-1} \rrbracket \nabla_{r}g_{L}\left(1\right)-\dfrac{1}{2J_{n}}\llbracket u_{e} \rrbracket \nabla_{r}g_{R}\left(1\right).
\end{empheq}
Therefore, in equation~\eqref{eq:jump q nabla 2} the gradient of the solution is simplified and we acquire
\begin{equation}\label{eq:Common point BR2 IP}
\Theta_{e}=-\frac{1}{4}\llbracket u_{e} \rrbracket\left[\llbracket u_{e} \rrbracket\left(\dfrac{1}{J_{n+1}}\nabla_{r}g_{L}\left(-1\right)-\dfrac{1}{J_{n}}\nabla_{r}g_{R}\left(1\right)\right)+\dfrac{1}{J_{n}}\llbracket u_{e-1} \rrbracket \nabla_{r}g_{L}\left(1\right) -\dfrac{1}{J_{n+1}}\llbracket u_{e+1} \rrbracket \nabla_{r}g_{R}\left(-1\right)\right]-\tau \llbracket u_{e} \rrbracket^{2}.
\end{equation}

We finally use Theorem~\ref{thm:correction function} and replace $\nabla_{r}g_{R}\left(1\right)$ and $\nabla_{r}g_{R}\left(-1\right)$ respectively by $-\nabla_{r}g_{L}\left(-1\right)$ and $-\nabla_{r}g_{L}\left(1\right)$ to obtain equation~\eqref{eq:last exact form}
\begin{equation*}
\Theta_{e}=-\llbracket u_{e} \rrbracket^{2}\left[\tau+\dfrac{1}{4}\nabla_{r}g_{L}\left(-1\right)\left(\dfrac{1}{J_{n}}+\dfrac{1}{J_{n+1}}\right)\right]-\dfrac{1}{4}\nabla_{r}g_{L}\left(1\right)\llbracket u_{e} \rrbracket\left(\dfrac{1}{J_{n}}\llbracket u_{e-1} \rrbracket+\dfrac{1}{J_{n+1}}\llbracket u_{e+1} \rrbracket\right).
\end{equation*}
\end{proof}

\begin{rmk}
We clearly see from~\eqref{eq:jump q nabla 2} that if we had taken the LDG scheme then $\nabla u_{e}$ would have been replaced by $q_{e}$. We would have then recovered the result of~\cite{castonguay_energy_2013}: $\Theta_{e}=-\tau\llbracket u_{e} \rrbracket^{2}$.
\end{rmk}

\begin{thm}\label{thm:stability IP}
Employing the IP scheme for the diffusion equation with the ESFR methods, $\tau$ greater than $\tau^{*}$ implies the energy stability, with
\normalfont
\begin{equation}
\tau^{*}=\frac{1}{2}\max\limits_{n\in\mathcal{T}_{h}}\left(\frac{1}{J_{n}}\right)\min\limits_{\kappa}\left(\left| \nabla_{r}g_{L}\left(1\right)\right|-\nabla_{r}g_{L}\left(-1\right)\right) \implies \dfrac{\mbox{d}}{\mbox{d}t}\Vert U\Vert_{p,c}^{2}\leq 0.
\end{equation}
\end{thm}

\begin{proof}
We apply the triangular inequality ($2ab\leq a^{2}+b^{2}$) with $a=\llbracket u_{e} \rrbracket$ and $b=\llbracket u_{e-1} \rrbracket$ or $b=\llbracket u_{e+1} \rrbracket$ to the last term of equation~\eqref{eq:last exact form}.
We also use the notation $J_{e,min}=\min\left(J_{n},J_{n+1}\right)$. As we do not know the sign of the quantity $\nabla_{r}g_{L}\left(1\right)$, we need to place it within an absolute value and the triangle inequality yields
\begin{equation}
\left\lbrace
\begin{array}{lll}
-\dfrac{1}{4J_{n}}\nabla_{r}g_{L}\left(1\right)\llbracket u_{e} \rrbracket \llbracket u_{e-1} \rrbracket
&\leq &
\dfrac{1}{8J_{e,min}} \left| \nabla_{r}g_{L}\left(1\right)\right| \left( \llbracket u_{e} \rrbracket^{2}+\llbracket u_{e-1} \rrbracket^{2}\right),\vspace{0.2cm}\\
-\dfrac{1}{4J_{n+1}}\nabla_{r}g_{L}\left(1\right)\llbracket u_{e} \rrbracket \llbracket u_{e+1} \rrbracket
&\leq &
\dfrac{1}{8J_{e,min}} \left| \nabla_{r}g_{L}\left(1\right)\right| \left( \llbracket u_{e} \rrbracket^{2}+\llbracket u_{e+1} \rrbracket^{2}\right).\vspace{0.2cm}
\end{array}
\right.
\end{equation}
Introducing these inequalities into~\eqref{eq:last exact form}, we retrieve
\begin{equation}
\Theta_{e}\leq -\llbracket u_{e} \rrbracket^{2}\left[\tau+\dfrac{1}{2J_{e,min}}\nabla_{r}g_{L}\left(-1\right)-\dfrac{1}{4J_{e,min}}\left| \nabla_{r}g_{L}\left(1\right)\right| \right]+\dfrac{1}{8J_{e,min}}\left| \nabla_{r}g_{L}\left(1\right)\right| \left(\llbracket u_{e-1} \rrbracket^{2}+\llbracket u_{e+1} \rrbracket^{2}\right).
\end{equation}
Next, we sum over all the edges, and for simplicity, we consider Dirichlet conditions (periodic boundaries conditions work as well). Therefore, on the first and last elements, the jump associated is equal to 0. We notice that $\Theta_{e}$ depends on the jump on edge $e$, $e+1$ and $e-1$. Therefore $\Theta_{e+1}$ and $\Theta_{e-1}$ will have a contribution of $\frac{1}{8}\left| \nabla_{r}g_{L}\left(1\right)\right|$ to the term $\llbracket u_{e} \rrbracket^{2}$ each. We then finally obtain the stability condition for the IP scheme
\begin{equation}
\displaystyle{
\sum_{e=1}^{N_{e}}\Theta_{e}\leq-\sum_{e=1}^{N_{e}}\llbracket u_{e} \rrbracket^{2}\left[\tau+\dfrac{1}{2J_{e,min}}\left(\nabla_{r}g_{L}\left(-1\right)-\left| \nabla_{r}g_{L}\left(1\right)\right|\right)\right].
}
\end{equation}

We may, thus conclude that to ensure energy stability for the diffusion equation in 1D for the IP scheme we require,
\begin{equation}
\displaystyle{
\tau\geq\tau^{*}=\frac{1}{2}\max_{n\in\mathcal{T}_{h}}\left(\frac{1}{J_{n}}\right)\left(\left| \nabla_{r}g_{L}\left(1\right)\right|-\nabla_{r}g_{L}\left(-1\right)\right) \implies \sum_{e=1}^{N_{e}}\Theta_{e}\leq 0 \implies\dfrac{\mbox{d}}{\mbox{d}t}\Vert U\Vert_{p,c}^{2}\leq 0.
}
\end{equation}

Now from the result of Corollary~\ref{cor:independent}, we know that the ESFR schemes do not depend on $\kappa$. Therefore the stability condition shouldn't depend on $\kappa$ either. We may sharpen the previous bound as,
\begin{equation}\label{eq:criterion IP}
\displaystyle{
\tau\geq\tau^{*}=\frac{1}{2}\max_{n\in\mathcal{T}_{h}}\left(\frac{1}{J_{n}}\right)\min\limits_{\kappa}\left(\left| \nabla_{r}g_{L}\left(1\right)\right|-\nabla_{r}g_{L}\left(-1\right)\right).
}
\end{equation}
\end{proof}

\begin{rmk}
Instead of considering a constant penalty term for all edges, we can find the stable penalty term $\tau_{e}^{*}$, for each edge. In the numerical simulations, $\tau$ would then be an array.

\begin{equation}\label{eq:criterion IP2}
\tau_{e}\geq\tau_{e}^{*}=\frac{1}{2}\left(\dfrac{1}{J_{e,-}}+\dfrac{1}{J_{e,+}}\right)\min\limits_{\kappa}\left(\left| \nabla_{r}g_{L}\left(1\right)\right|-\nabla_{r}g_{L}\left(-1\right)\right),\,\, \forall e \in\llbracket 1,N_{e}\rrbracket \implies \dfrac{\mbox{d}}{\mbox{d}t}\Vert U\Vert_{p,c}^{2}\leq 0.
\end{equation}
\end{rmk}

\begin{rmk}
Trivial derivations showed that $\min\limits_{\kappa}\left(\left| \nabla_{r}g_{L}\left(1\right)\right|-\nabla_{r}g_{L}\left(-1\right)\right)$ is obtained for $\kappa_{min}=\frac{2\left(p+1\right)}{p\left(2p+1\right)\left(a_{p}p!\right)^{2}}$ or any greater value, where $a_{p}$ is given in Appendix~\ref{sec:ESFR correction function}. Numerical simulations were conducted to verify this result and are represented in Figure~\ref{fig:minimum g}. In the rest of the article $\tau_{theory}^{*}$ will be computed with this optimal value.

\begin{figure}[H]
\centering
\includegraphics[width=4in]{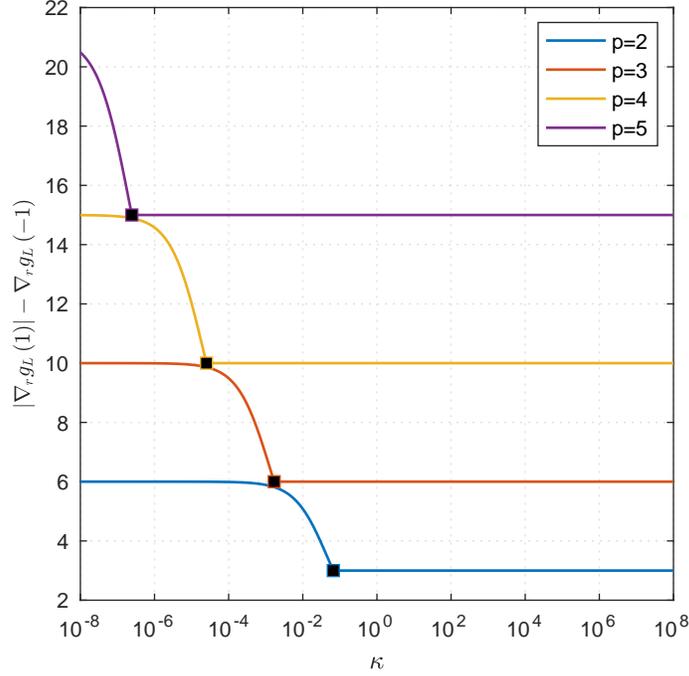}
\caption{$\left| \nabla_{r}g_{L}\left(1\right)\right|-\nabla_{r}g_{L}\left(-1\right)$ for different values of $\kappa$ and $p$. The black squares denotes the value obtained with $\kappa_{min}$.}
\label{fig:minimum g}
\end{figure}
\end{rmk}
\subsection{Numerical results}\label{sec:IP numerical experiment}

Several numerical simulations have been conducted to verify the result~\eqref{eq:criterion IP}. For equally spaced elements $\left(J_{n}=J\right)$ equations \eqref{eq:main eq compu tot} and \eqref{eq:aux eq compu tot} can be written as follows for the IP scheme
\begin{equation}\label{system numeric}
\left\lbrace
\begin{array}{lll}
\dfrac{\partial u}{\partial t}&=&\dfrac{b}{J}\left(\nabla_{r} q+\left.\left(\dfrac{1}{J}\left\lbrace\left\lbrace\nabla_{r} u\right\rbrace\right\rbrace -q-\tau\llbracket u_{e}\rrbracket\right)\right|_{-1}\nabla_{r}h_{L}+\left.\left(\dfrac{1}{J}\left\lbrace\left\lbrace\nabla_{r} u\right\rbrace\right\rbrace -q-\tau\llbracket u_{e}\rrbracket\right)\right|_{1} \nabla_{r}h_{R}\right),\vspace{0.2cm}\\
q&=&\dfrac{1}{J}\left(\nabla_{r}u+\left.\left(\left\lbrace\left\lbrace u\right\rbrace\right\rbrace -u\right)\right|_{-1}\nabla_{r}g_{L}+\left.\left(\left\lbrace\left\lbrace u\right\rbrace\right\rbrace -u\right)\right|_{1}\nabla_{r}g_{R}\right).
\end{array}
\right.
\end{equation}
The purpose of this numerical simulation is to determine, numerically, the minimum penalization term $\tau$ required for stability. The problem solved is the following: find $u\left(x,t\right)$ such that
\begin{equation}\label{eq:PDE}
\left\lbrace
\begin{array}{llll}
\dfrac{\partial u}{\partial t}&=&b\Delta u,&\,\text{for}\, x\in\left[0;2\pi\right]\, \text{and}\, t\in\left[0,2\right],\vspace{0.2cm}\\
u\left(x,0\right)&=&\sin\left(x\right)+\cos\left(x\right).
\end{array}
\right.
\end{equation}
We impose Dirichlet boundary conditions and the exact solution for this system is $u_{exact}=e^{-bt}\left(\sin\left(x\right)+\cos\left(x\right)\right)$.

In order to assess the stability of the numerical experiment, we choose a criterion on the upper bound of the solution such that $\left|u\left(x,t\right)\right|\leq u_{max}$, where $u_{max}$ is an arbitrarily large and problem dependent value. Here $u_{exact}\leq \sqrt{2}$ for all $t$, we hence choose $u_{max}=2$, i.e, sufficiently greater than $\sqrt{2}$, for the differences to be caused by numerical instabilities arising from the scheme. The procedure is as follows:


\begin{algorithm}[H]
\SetAlgoLined
\KwResult{Obtain $\tau_{numerical}^{*}$ with a precision of 0.01}
$i\leftarrow 0$\;
$\mbox{d}\tau_{i}\leftarrow 1$\;
$j\leftarrow 0$\;
$\tau_{j}\leftarrow\tau$\;
\tcc{Choose $\tau$ to be sufficiently low such that the first iteration is unstable. For the IP scheme, a negative value is sufficient (which is not the case for LDG: a negative penalty term may still provide a stable scheme).}
	\While{$i<3$}{
		Compute $u\left(x,t\right)$ for $x\in\left[0,2\pi\right]$ and $t\in\left[0,2\right]$. The final time $t=2$ is chosen to ensure that a few thousands iterations are run\;
		\eIf{$\max\left|u\left(x,t\right)\right|\geq u_{max}$}{
		$\tau_{j+1}\leftarrow\tau_{j}+\mbox{d}\tau_{i}$\;
		$j\leftarrow j+1$\;
		}{
		$\tau_{j+1}\leftarrow\tau_{j}-\mbox{d}\tau_{i}$\;
		$\mbox{d}\tau_{i}\leftarrow\mbox{d}\tau_{i}/10$\;
		$i\leftarrow i+1$\;
		}
	}
	\caption{Minimal value of $\tau$ ($\tau_{numerical}^{*}$) to be stable}
	\label{alg:algorithm IP}	
\end{algorithm}

Through this Algorithm~\ref{alg:algorithm IP} we achieve the minimal numerical penalization term: $\tau_{numerical}^{*}$ with a precision of $\mbox{d}\tau_{2}=0.01$ for which we observe convergence. It is generally accepted that for $\tau\geq\tau_{numerical}^{*}$ we have stability but we may lose the performance of the method (conditioning of the matrix~\cite{shahbazi_explicit_2005} and increased maximal time step Section~\ref{sec:Von neumann analysis}), if $\tau$ is too large.

We have several choices of parameters to conduct the simulations. For the results presented, we choose a diffusive parameter $b=1$. The domain was discretized into $N_{K}=32$ equally-distributed elements ($J$ is a constant throughout the domain). The time resolution solver is a five stage, fourth order Runge-Kutta method (RK54 introduced by Carpenter and Kennedy \cite{carpenter_fourth-order_1994}). Since the purpose of this study is to find the minimal penalty term $\tau$, we cannot immediately conduct a von Neumann analysis to find the maximal time step $\Delta t_{max}$. Therefore, we select a time step recommended by Hesthaven and Warburton \cite{NDG}:
$\Delta t=\frac{1}{b}CFL\min\left(\Delta x\right)^{2}$ (in our cases $CFL=0.05$). In Section~\ref{sec:Von neumann analysis}, we conduct a von Neumann analysis and find the maximum time step associated to a value of $\tau$.

We conduct several experiments with the degree of the interpolation polynomial $p$ as a parameter and employ the LGL points.

\begin{table}[H]
\centering
\makebox[\textwidth][c]{
\begin{tabular}{|c||c|c|c|c|c||c|c|c|c|c|}
\hline
$p$&\multicolumn{5}{c||}{2}&\multicolumn{5}{c|}{3}\\
\hline
\bf \backslashbox{$c$}{$\kappa$}&$\kappa_{DG}$&$\kappa_{SD}$&$\kappa_{Hu}$&$\kappa_{+}$&$10^{5}$&$\kappa_{DG}$&$\kappa_{SD}$&$\kappa_{Hu}$&$\kappa_{+}$&$10^{5}$ \\
\hline
$\tau_{theory}^{*}$&15.28&15.28&15.28&15.28&15.28&30.56&30.56&30.56&30.56&30.56\\
\hline
\hline
$c_{DG}$&15.21&15.21&15.21&15.21&15.21&30.49&30.49&30.49&30.49&30.49\\
\hline
$c_{SD}$&15.18&15.18&15.18&15.18&15.18&30.46&30.46&30.46&30.46&30.46\\
\hline
$c_{HU}$&15.12&15.12&15.12&15.12&15.12&30.43&30.43&30.43&30.43&30.43\\
\hline
$c_{+}$&14.97&14.97&14.68&14.97&14.97&30.35&30.35&30.35&30.35&30.35\\
\hline
$10^{5}$&5.02&5.02&5.02&5.02&5.02&15.21&15.21&15.21&15.21&15.21\\
\hline
\end{tabular}
}
\caption{$\tau_{theory}^{*}$ compared to $\tau_{numerical}^{*}$ for the IP scheme for p=2 and p=3 for 32 elements.}
\label{tab:result IP Np=3}
\end{table}

As we observed in~\eqref{eq:criterion IP} $\tau_{theory}^{*}$ depends only on the Jacobian $J$ and on the degree of polynomial $p$. Hence it is independent of $\kappa$ and $c$. The first line of entries in Table~\ref{tab:result IP Np=3} corresponds to these theoretical values. The rest of the values corresponds to the numerical values obtained via Algorithm~\ref{alg:algorithm IP}. Two main observations can be made concerning the results in Table~\ref{tab:result IP Np=3}. First, we discern that the relationship $\tau_{theory}^{*}\geq\tau_{numerical}^{*}$ is respected. If we take any value of $\tau$ greater than the theoretical result, then we are guaranteed to achieve convergence. Second, as $c$ decreases, $\tau_{numerical}^{*}$ approaches $\tau_{theory}^{*}$. The fact that $\tau_{numerical}^{*}$ is less than our theoretical results comes from our stability condition. We choose $\tau_{theory}^{*}$ such that $\sum_{e=}^{N_{e}}\Theta_{e}\leq 0$ in order to ensure: $\dfrac{\mbox{d}}{\mbox{d}t}\Vert U\Vert_{p,c}^{2}\leq -b\Vert Q \Vert _{p,\kappa}^{2}+b\sum_{e=1}^{N_{e}}\Theta_{e}\leq 0$. In fact, a sufficient criterion for stability is just $\dfrac{\mbox{d}}{\mbox{d}t}\Vert U\Vert_{p,c}^{2} \leq 0$.

Therefore, to improve our results, we could have looked for a bound for $\tau$ such that
\begin{equation}
\displaystyle{
\sum_{e=1}^{N_{e}}\Theta_{e}\leq \Vert Q\Vert_{p,\kappa}^{2}
}.
\end{equation}
The result, $\tau_{theory}^{*}\geq\tau_{numerical}^{*}$, confirms that there is, yet to find, a sharper inequality for $\tau$. Moreover the norm of $u$ chosen in~\eqref{eq:norm ESFR} is stronger than just the energy stability. Indeed proving $\dfrac{\mbox{d}}{\mbox{d}t}\Vert u\Vert_{0,\Omega}^{2}\leq 0$ is sufficient to ensure stability.
\section{BR2 stability condition}\label{sec:section 2 BR2}
The purpose of this section is to prove the equivalence between the IP and the BR2 numerical fluxes in one-dimensional problems. This result was suggested by Wang et al~\cite{hybrid_meshes}.
\subsection{Numerical insight of the lifting operator}\label{sec:numerical BR2}

We derive a BR2 stability condition through a similar approach. Recall from~\eqref{eq:numerical fluxes} that the numerical fluxes are
\begin{equation}\label{eq:BR2 numerical flux}
\left\lbrace
\begin{array}{lll}
u^{*}_{e}&=&\left\lbrace\left\lbrace u_{e}\right\rbrace\right\rbrace,\vspace{0.1cm}\\
q^{*}_{e}&=&\left\lbrace\left\lbrace \nabla u\right\rbrace\right\rbrace +s\, \left\lbrace\left\lbrace r^{e}\left(\llbracket u\rrbracket\right)\right\rbrace\right\rbrace,
\end{array}
\right.
\end{equation}
where the lifting operator $r^{e}$ is defined in 1D in~\eqref{eq:def re}.

Throughout this article, a bold uppercase character, $\mathbf{A}$, signifies a matrix, a bold lowercase character, $\mathbf{a}$, denotes a column vector. With equation~\eqref{eq:def re}, we see that there exists a link between the BR2 and IP schemes. The following presents a numerical formulation such that: $\left\lbrace\left\lbrace r^{e}\left(\llbracket u\rrbracket\right)\right\rbrace\right\rbrace=-f\left(p,J_{n},J_{n+1}\right)\llbracket u_{e}\rrbracket$ with $f$ a scalar function depending on $p$, the degree of the polynomial, $J_{n}$ and $J_{n+1}$, the Jacobians associated respectively to elements $\Omega_{n}$ and $\Omega_{n+1}$. Numerically, we can introduce the definition of the lifting operator in~\eqref{eq:def re} through the relationship,
\begin{equation}
\mathbf{M}\,\mathbf{\hat{r}}^{e}\left(\llbracket u\rrbracket\right)=-\Phi \llbracket u_{e}\rrbracket.
\end{equation}
We can retrieve the coefficients of the correction term as
 \begin{equation}\label{eq:invert mass}
\mathbf{\hat{r}}^{e}\left(\llbracket u\rrbracket\right)=-\mathbf{M}^{-1}\Phi \llbracket u_{e}\rrbracket,
\end{equation}
where, $\mathbf{M}$ denotes the mass matrix of the entire domain of size $N_{K}N_{p}\times N_{K}N_{p}$. We have: $m_{ij}=\int_{\Omega_{n}}\phi_{i}\phi_{j}\mbox{d}x$, $\mathbf{\hat{r}}^{e}\left(\llbracket u\rrbracket\right)$ is the column vector of the correction term values, $r^{e}\left(\llbracket u\rrbracket\right)$. $\Phi$ defines the column vector associated with $\left\lbrace\left\lbrace \phi\right\rbrace\right\rbrace _{e}$. Therefore this vector has for support the elements which have the edge $e$: $Supp\left(\Phi\right)=\Omega_{n}\cup \Omega_{n+1}$. Lastly, $\llbracket u_{e}\rrbracket$ is a scalar value: the jump of the solution on the edge $e$.

Since $\mathbf{M}$ is a diagonal block and the support of $\Phi$ is $\Omega_{n}\cup \Omega_{n+1}$ $\left(Supp\left(\Phi\right)=\Omega_{n}\cup \Omega_{n+1}\right)$, we observe that $\mathbf{\hat{r}}^{e}\left(\llbracket u\rrbracket\right)$ has for mathematical support the elements bordering the edge $e$: $Supp\left(r^{e}\left(\llbracket u\rrbracket\right)\right)=\Omega_{n}\cup \Omega_{n+1}$. We can interpret the correction $r^{e}\left(\llbracket u\rrbracket\right)$ as a polynomial of degree $p$ on each element $\Omega_{n}$ and $\Omega_{n+1}$. With equation~\eqref{eq:invert mass}, we can retrieve the coeffcients of $r^{e}\left(\llbracket u\rrbracket\right)$ on its support. To compute the numerical flux $q^{*}$, we just require the average of $r^{e}\llbracket u\rrbracket$ on the edge $e$. We use Lagrange polynomials both as our test functions and as our interpolation functions, corresponding thus to a DG formulation.

\begin{equation}
\begin{array}{lll}
\left\lbrace\left\lbrace r^{e}\left(\llbracket u\rrbracket\right)\right\rbrace\right\rbrace&=&\dfrac{1}{2}\left(r^{e}\left(\llbracket u \rrbracket\right)_{e,-}+r^{e}\left(\llbracket u \rrbracket\right)_{e,+}\right)\\
&=&\tiny
-\dfrac{\llbracket u_{e} \rrbracket}{4}
\begin{bmatrix}
l_{1}\left(1\right)&\dots&l_{Np}\left(1\right)&l_{1}\left(-1\right)&\dots& l_{Np}\left(-1\right)
\end{bmatrix}
\begin{bmatrix}
\mathbf{M}_{\Omega_{n}}^{-1}&\mathbf{0}\\
\mathbf{0}&\mathbf{M}_{\Omega_{n+1}}^{-1}
\end{bmatrix}
\begin{bmatrix}
l_{1}\left(1\right)\\ \vdots\\ l_{Np}\left(1\right)\\ l_{1}\left(-1\right)\\ \vdots \\ l_{Np}\left(-1\right)
\end{bmatrix}\\\normalsize
&=&-f\left(p,J_{n},J_{n+1}\right)\llbracket u_{e}\rrbracket,
\end{array}
\label{eq:f link to re}
\end{equation}
where $f$ is a factor enabling an equivalency between the BR2 and the IP schemes and it is tabulated in Table~\ref{tab:Table factor} for a range of polynomial degrees on 16 elements. $f$ does not depend on the type of nodes.

\begin{table}[H]
\centering
\makebox[\textwidth][c]{
\begin{tabular}{|c|c|c|c|c|c|c|c|c|}
\hline
$ p$&1&2&3&4&5&6&7&8\\
\hline
$f$&5.093& 11.46& 20.37& 31.83 & 45.84& 62.39& 81.49& 103.1\\
\hline
\end{tabular}
}
\caption{Factor between BR2 and IP for Equidistant, Lobatto-Gauss-Legendre and Gauss-Legendre points for 16 elements. The element are equally-distributed ($J$ is a constant throughout the domain), the length of the entire domain is $2\pi$.}
\label{tab:Table factor}
\end{table}

This result is of the utmost importance for the analogy. It signifies that we have: $q^{*}_{IP}= q^{*}_{BR2}\Leftrightarrow \tau_{e}=s_{e}\,f_{n}\left(p,J_{n},J_{n+1}\right)$.

\subsection{Theoretical equivalence between the BR2 and the IP scheme}
Many ESFR researchers have established the equivalence between the DG and ESFR methods with $c=0$. Huynh established, \cite{huynh_high-order_2011}, a direct link between the derivative of the correction function and the lifting operator applied to recover the FR formulation from the DG formulation. The following presents an analogy between the lifting operator $r^{e}$ and the correction function, allowing to define an explicit formulation for $f$.

Let us consider the linear advection problem in one dimension.
\begin{equation}
\dfrac{\partial u}{\partial t}=-\nabla F,
\end{equation}
where $F=au$ is the flux and $a$ the velocity. Using the mapping defined in~\eqref{eq:affine mapping} and applying the DG method, in its strong formulation \cite{NDG}, results in
\begin{equation}\label{eq:DG formulation advection}
\displaystyle{
\sum_{n}\int_{-1}^{1}\dfrac{\partial u_{n}}{\partial t}\phi\,\mbox{d}r=-\left(\int_{-1}^{1}\nabla_{r} F_{n}\phi\,\mbox{d}r+\underbrace{\left[\left(F_{n}^{*}-F_{n}\right)\left(1\right)\phi\left(1\right)-\left(F_{n}^{*}-F_{n}\right)\left(-1\right)\phi\left(-1\right)\right]}_{\text{Terms to be lifted}}\right),
}
\end{equation}
where $\phi$ is the test function, polynomial of degree less or equal to $p$ and $F^{*}$ is the numerical flux for the advection equation.
Introducing $g_{L}$ and $g_{R}$, such that
\begin{eqnarray}
\stretchint{4ex}_{\bs\bs -1}^{1}\nabla_{r}g_{L}\phi\,\mbox{d}r&=&-\phi\left(-1\right),\label{eq:gL integral}\\
\stretchint{4ex}_{\bs\bs-1}^{1}\nabla_{r}g_{R}\phi\,\mbox{d}r&=&\phi\left(1\right),\label{eq:gR integral}
\end{eqnarray}
we then obtain from~\eqref{eq:DG formulation advection},
\begin{equation}
\displaystyle{
\sum_{n}\int_{-1}^{1}\left[\left(\dfrac{\partial u_{n}}{\partial t}+\nabla_{r} F_{n}+\left(F_{n}^{*}-F_{n}\right)\left(-1\right)\nabla_{r}g_{L}+\left(F_{n}^{*}-F_{n}\right)\left(1\right)\nabla_{r}g_{R}\right)\phi\right]=0.
}
\end{equation}

We have, for all elements $\Omega_{n}$, the differential equation (FR method),
\begin{equation}\label{eq:FR formulation advection}
\dfrac{\partial u_{n}}{\partial t}=-\left[\nabla_{r} F_{n}+\left(F_{n}^{*}-F_{n}\right)\left(-1\right)\nabla_{r}g_{L}+\left(F_{n}^{*}-F_{n}\right)\left(1\right)\nabla_{r}g_{R}\right].
\end{equation}

As explained in Section~\ref{sec:Preliminaries}, both $g_{L}$ and $g_{R}$ are polynomials of degree $p+1$, hence $p+2$ conditions are required to define them properly.

\begin{lem}\label{lem:lifting FR}
Employing the ESFR correction functions defined in Definitions~\ref{def:Left} and~\ref{def:right}, there exists a unique value of $\kappa$ such that equations~\eqref{eq:gL integral} and~\eqref{eq:gR integral} are valid and this value is $\kappa=0$.
\end{lem}

\begin{proof}
For brevity, only the proof for $g_{L}$ (equation~\eqref{eq:gL integral}) will be given. The derivations for $g_{R}$ are similar.

Integrating by parts equation~\eqref{eq:gL integral},
\begin{equation}\label{eq:equivalence eq31}
\underbrace{g_{L}\left(1\right)\phi\left(1\right)}_{\text{Term A}}-\underbrace{g_{L}\left(-1\right)\phi\left(-1\right)}_{\text{Term B}}-\underbrace{\stretchint{4ex}_{\bs\bs -1}^{1}g_{L}\nabla_{r}\phi\,\mbox{d}r}_{\text{Term C}}=-\phi\left(-1\right).
\end{equation}

If Term A and Term C are equal to 0 and $g_{L}\left(-1\right)=1$ then equation~\eqref{eq:equivalence eq31} is valid.

Employing the left-ESFR correction function, we have $g_{L}\left(-1\right)=1$ and $g_{L}\left(1\right)=0$ (Term A=0). As for Term C, we have
\begin{equation}
\stretchint{4ex}_{\bs\bs -1}^{1}g_{L}\nabla_{r}\phi\,\mbox{d}r=\dfrac{\left(-1\right)^{p}}{2}\left[\stretchint{4ex}_{\bs\bs -1}^{1}\Psi_{p}\nabla_{r}\phi\,\mbox{d}r-\dfrac{\eta_{p,\kappa}}{1+\eta_{p,\kappa}}\stretchint{4ex}_{\bs\bs -1}^{1}\Psi_{p-1}\nabla_{r}\phi\,\mbox{d}r-\dfrac{1}{1+\eta_{p,\kappa}}\stretchint{4ex}_{\bs\bs -1}^{1}\Psi_{p+1}\nabla_{r}\phi\,\mbox{d}r\right].
\end{equation}

$\nabla_{r}\phi$ is a polynomial of degree less or equal to $p-1$ and hence can be decomposed in Legendre polynomials, $\phi=\mathlarger{\sum}_{i=0}^{p-1}\alpha_{i}\Psi_{i}$. Using \ref{prop:integral legendre},
\begin{equation}
\stretchint{4ex}_{\bs\bs -1}^{1}g_{L}\nabla_{r}\phi\,\mbox{d}r=\dfrac{\left(-1\right)^{p+1}\eta_{p,\kappa}\alpha_{p-1}}{\left(1+\eta_{p,\kappa}\right)\left(2p-1\right)},
\end{equation}
which is equal to 0 if $\eta_{p,\kappa}=0$ and  hence if $\kappa=0$.

\end{proof}

\begin{rmk}
With $\kappa=0$, the FR formulation in equation~\eqref{eq:FR formulation advection} is equivalent to the DG formulation in equation~\eqref{eq:DG formulation advection}. Therefore, for $\kappa=0$, the left-ESFR correction function (resp. right) is denoted $g_{L,DG}$ (resp. $g_{R,DG}$).
\end{rmk}
\begin{rmk}
Huynh, \cite{huynh_high-order_2011}, did not use ESFR correction functions. He created $g_{L}$ such that $g_{L}\left(-1\right)=1$ ($1$ condition) and $g_{L}\left(1\right)=0$ ($1$ condition). To recover the DG formulation the remaining requirement was $g_{L}$ orthogonal to $\mathbb{P}_{p-1}$ ($p$ conditions). He extended the FR formulation by relaxing the condition on Term C into $g_{L}$ othogonal to $\mathbb{P}_{p-2}$ and thus obtain a large class of methods ($1$ free parameter). Vincent et al. reduced this class and defined the ESFR methods \cite{vincent_new_2011} (see Definition~\ref{def:Left} and~\ref{def:right}) which ensure the stability for the advection problem. Therefore for ESFR methods and $\kappa\neq 0$, the FR formulation in~\eqref{eq:FR formulation advection} is not equivalent to the DG formulation in~\eqref{eq:DG formulation advection}. However it has been shown \cite{zwanenburg_equivalence_2016}, that the ESFR methods can be cast into a filtered DG method. 
\end{rmk}

An analogy between the correction function and the lifting operator, $r^{e}$, is now presented. Section~\ref{sec:numerical BR2} made us aware that the lifting operator $r^{e}$, for an edge $e$, is defined on both $\Omega_{n}$ and $\Omega_{n+1}$ and is a polynomial of degree $p$ on both $\Omega_{n}$ and $\Omega_{n+1}$ (see Figure~\ref{fig:edges_elements}). We define the space $\Omega_{e}=\bigcup\limits_{i=n}^{n+1}\Omega_{i}$. Similar to the affine mapping defined in~\eqref{eq:affine mapping}, we define the surjection
\begin{equation}\label{eq:affine mapping re}
\begin{array}{lccl}
\mathcal{M}^{-1}_{e}\colon &\Omega_{e}&\to&\left[-1 , 1\right]\\
&x&\mapsto&\dfrac{\left(2x-\left(x_{n}+x_{n+1}\right)\right)}{h_{n}}\left.\chi\right|_{\Omega_{n}}\left(x\right)+\dfrac{\left(2x-\left(x_{n+1}+x_{n+2}\right)\right)}{h_{n+1}}\left.\chi\right|_{\Omega_{n+1}}\left(x\right),
\end{array}
\end{equation}
where $\left.\chi\right|_{\Omega_{i}}\left(x\right)$ is equal to 1 if $x\in\Omega_{i}$, or 0 if $x\not\in\Omega_{i}$.

\begin{lem}\label{lem:formula re correction function}
The lifting operator, $r^{e}$, employed in the BR2 scheme, is equivalent to the following formula
\begin{equation}\label{eq:formula re correction function}
r^{e}\left(\llbracket u \rrbracket\right)\left(x\right)=\dfrac{\llbracket u_{e}\rrbracket}{2}\left[-\dfrac{1}{J_{n}}\nabla_{r}g_{R,DG}\left(\mathcal{M}^{-1}_{e}\left(x\right)\right))\left.\chi\right|_{\Omega_{n}}+\dfrac{1}{J_{n+1}}\nabla_{r}g_{L,DG}\left(\mathcal{M}^{-1}_{e}\left(x\right)\right)\left.\chi\right|_{\Omega_{n+1}}\right],
\end{equation}
where $\nabla_{r}g_{L,DG}$ (resp. $\nabla_{r}g_{R,DG}$) is the derivative with $r$ of the left-ESFR (resp. right-ESFR) correction function for $\kappa=0$.
\end{lem}

\begin{proof}
The lifting operator $r^{e}$ is a polynomial of degree $p$ on each element forming $\Omega_{e}$.
Using the definition of the lifting operator in~\eqref{eq:def re},
\begin{equation}
\begin{array}{lll}
\stretchint{4ex}_{\bs\bs \Omega}r^{e}\left(\llbracket u\rrbracket\right)\phi\,\mbox{d}x&=&\stretchint{4ex}_{\bs\bs \Omega_{e}}r^{e}\left(\llbracket u\rrbracket\right)\phi\,\mbox{d}x\vspace{0.2cm}\\
&=&-\llbracket u_{e}\rrbracket\left\lbrace\left\lbrace \phi\right\rbrace\right\rbrace _{e}\vspace{0.2cm}\\
&=&-\dfrac{\llbracket u_{e} \rrbracket}{2}\left[\phi_{e,-}+\phi_{e,+}\right]\vspace{0.2cm}\\
&=&-\dfrac{\llbracket u_{e} \rrbracket}{2}\left[\left.\phi\left(\mathcal{M}^{-1}_{e}\left(x_{n+1}\right)\right)\right|_{\Omega_{n}}+\left.\phi\left(\mathcal{M}^{-1}_{e}\left(x_{n+1}\right)\right)\right|_{\Omega_{n+1}}\right],
\end{array}
\end{equation}
where $\left.\phi\right|_{\Omega_{n}}$ (resp. $\left.\phi\right|_{\Omega_{n+1}}$) is the test function on $\Omega_{n}$ (resp. $\Omega_{n+1}$).

Employing Lemma~\ref{lem:lifting FR},
\begin{equation}
\begin{array}{lll}
\stretchint{4ex}_{\bs\bs \Omega_{e}}r^{e}\left(\llbracket u\rrbracket\right)\phi\,\mbox{d}x&=&-\dfrac{\llbracket u_{e} \rrbracket}{2}\left[\stretchint{4ex}_{\bs\bs -1}^{1}\nabla_{r}g_{R,DG}\left.\phi\right|_{\Omega_{n}}\,\mbox{d}r-\stretchint{4ex}_{\bs\bs -1}^{1}\nabla_{r}g_{L,DG}\left.\phi\right|_{\Omega_{n+1}}\,\mbox{d}r\right]\vspace{0.2cm}\\
&=&-\dfrac{\llbracket u_{e} \rrbracket}{2}\left[\dfrac{1}{J_{n}}\stretchint{4ex}_{\bs\bs \Omega_{n}}\nabla_{r}g_{R,DG}\left.\phi\right|_{\Omega_{n}}\,\mbox{d}x-\dfrac{1}{J_{n+1}}\stretchint{4ex}_{\bs\bs \Omega_{n+1}}\nabla_{r}g_{L,DG}\left.\phi\right|_{\Omega_{n+1}}\,\mbox{d}x\right]\vspace{0.2cm}\\
&=&-\dfrac{\llbracket u_{e} \rrbracket}{2}\stretchint{5ex}_{\bs\bs \Omega_{e}}\left(\dfrac{1}{J_{n}}\nabla_{r}g_{R,DG}\left.\chi\right|_{\Omega_{n}}-\dfrac{1}{J_{n+1}}\nabla_{r}g_{L,DG}\left.\chi\right|_{\Omega_{n+1}}\right)\phi\,\mbox{d}x,
\end{array}
\end{equation}
where $\phi=\left.\phi\right|_{\Omega_{n}}\left.\chi\right|_{\Omega_{n}}+\left.\phi\right|_{\Omega_{n+1}}\left.\chi\right|_{\Omega_{n+1}}$.

We gather all the terms in the same integral and we factor out the test function $\phi$,
\begin{equation}
\stretchint{5ex}_{\bs\bs\Omega_{e}}\left(r^{e}\left(\llbracket u\rrbracket\right)+\dfrac{\llbracket u_{e}\rrbracket} {2}\left(\dfrac{1}{J_{n}}\nabla_{r}g_{R,DG}\left.\chi\right|_{\Omega_{n}}-\dfrac{1}{J_{n+1}}\nabla_{r}g_{L,DG}\left.\chi\right|_{\Omega_{n+1}}\right)\right)\phi\,\mbox{d}x=0.
\end{equation}

This relation is true for any test function $\phi$ on $\Omega_{e}$, hence retrieving equation~\eqref{eq:formula re correction function}.
\end{proof}

\begin{thm}\label{thm:equivalence BR2 IP}
The BR2 scheme is equivalent to the IP scheme, both defined in~\eqref{eq:numerical fluxes}, if and only if
\begin{equation}
s=\dfrac{8 J_{n}J_{n+1}}{\left(J_{n}+J_{n+1}\right)\left(p+1\right)^{2}}\tau.
\end{equation}
\end{thm}

\begin{proof}
The only difference between the IP and the BR2 schemes is the penalty term; $-\tau\llbracket u \rrbracket$ for the IP scheme and $s\left\lbrace\left\lbrace r^{e}\left(\llbracket u \rrbracket\right)\right\rbrace\right\rbrace$ for the BR2 scheme. Lemma~\ref{lem:formula re correction function} yields,
\begin{equation}
\begin{array}{lll}
s\left\lbrace\left\lbrace r^{e}\left(\llbracket u \rrbracket\right)\right\rbrace\right\rbrace&=&\dfrac{s}{2}\left(r^{e}\left(\llbracket u \rrbracket\right)_{e,-}+r^{e}\left(\llbracket u \rrbracket\right)_{e,+}\right)\vspace{0.2cm}\\
&=&\dfrac{s\llbracket u \rrbracket}{4}\left(-\dfrac{1}{J_{n}}g_{R,DG}^{'}\left(1\right)+\dfrac{1}{J_{n+1}}g_{L,DG}^{'}\left(-1\right)\right).
\end{array}
\end{equation}

We use Theorem~\ref{thm:correction function} ($g_{R,DG}^{'}\left(1\right)=-g_{L,DG}^{'}\left(-1\right)$) and then employ~\ref{prop:derivative value +1 -1} of the Legendre polynomial. It yields $g_{L,DG}^{'}\left(-1\right)=-\dfrac{\left(p+1\right)^{2}}{2}$ and hence

\begin{equation}
s\left\lbrace\left\lbrace r^{e}\left(\llbracket u \rrbracket\right)\right\rbrace\right\rbrace=-\dfrac{s\llbracket u_{e} \rrbracket\left(p+1\right)^{2}}{8}\left(\dfrac{1}{J_{n}}+\dfrac{1}{J_{n+1}}\right).
\end{equation}

We finally have
\begin{equation}
\begin{array}{llrll}
&&q^{*}_{e,BR2}&=&q^{*}_{e,IP}\vspace{0.2cm}\\
&\Longleftrightarrow&s\left\lbrace\left\lbrace r^{e}\left(\llbracket u \rrbracket\right)\right\rbrace\right\rbrace&=&-\tau\llbracket u_{e} \rrbracket\vspace{0.2cm}\\
&\Longleftrightarrow&-\dfrac{s\llbracket u_{e} \rrbracket\left(p+1\right)^{2}}{8}\left(\dfrac{1}{J_{n}}+\dfrac{1}{J_{n+1}}\right)&=&-\tau\llbracket u_{e} \rrbracket\vspace{0.2cm}\\
&\Longleftrightarrow&s&=&\dfrac{8 J_{n}J_{n+1}}{\left(J_{n}+J_{n+1}\right)\left(p+1\right)^{2}}\tau.
\end{array}
\end{equation}
\end{proof}

\begin{cor}
Employing the BR2 scheme for the diffusion equation with the ESFR methods, $s$ greater than $s^{*}$ implies the energy stability, with
\normalfont
\begin{equation}\label{eq:criterion BR2}
\displaystyle{
s^{*}=\dfrac{2\min\limits_{\kappa}\left(\left|\nabla_{r}g_{L}\left(1\right)\right|-\nabla_{r}g_{L}\left(-1\right)\right)}{\left(p+1\right)^{2}} \implies \sum_{e=1}^{N_{e}}\Theta_{e}\leq 0 \implies\dfrac{\mbox{d}}{\mbox{d}t}\Vert U\Vert_{p,c}^{2}\leq 0.
}
\end{equation}
\end{cor}

\begin{proof}
Combining Theorems~\ref{thm:equivalence BR2 IP} and~\ref{thm:stability IP}, we show that the BR2 scheme is equivalent to the IP scheme, and then we use the criterion of stability found for the IP scheme.
\end{proof}

\begin{rmk}
From Theorem~\ref{thm:equivalence BR2 IP}, the factor $f$, defined in Section~\ref{sec:numerical BR2} can be expressed as,
\begin{equation}\label{eq:expression f}
f\left(p,J_{n},J_{n+1}\right)=\dfrac{1}{8}\left(p+1\right)^{2}\left(\dfrac{1}{J_{n}}+\dfrac{1}{J_{n+1}}\right).
\end{equation}
We can easily verify that~\eqref{eq:expression f} verifies Table~\ref{tab:Table factor}, where, for the Jacobian $\left(\text{equation~\eqref{eq:def jacobian}: } J=cste=\frac{2\pi}{2N_{k}}\right)$, the factor $f=\dfrac{1}{4\pi}\left(p+1\right)^{2}N_{K}$.
\end{rmk}

\subsection{Numerical results}
We launch the same suite of test cases as for the IP case, Section~\ref{sec:IP numerical experiment}.
Due to the equivalency between the IP and BR2 in 1D, we indeed obtain the same results as Table~\ref{tab:result IP Np=3} multiplied by the adapted factor $f$ in Table~\ref{tab:Table factor}. These BR2 results are presented in Table~\ref{tab:result BR2 Np=3}.

\begin{table}[H]
\centering
\makebox[\textwidth][c]{
\begin{tabular}{|c||c|c|c|c|c||c|c|c|c|c|}
\hline
p&\multicolumn{5}{c||}{2}&\multicolumn{5}{c|}{3}\\
\hline
\bf \backslashbox{$c$}{$\kappa$}&$\kappa_{DG}$&$\kappa_{SD}$&$\kappa_{Hu}$&$\kappa_{+}$&$10^{5}$&$\kappa_{DG}$&$\kappa_{SD}$&$\kappa_{Hu}$&$\kappa_{+}$&$10^{5}$ \\
\hline
$s_{theory}^{*}$&0.67&0.67&0.67&0.67&0.67&0.75&0.75&0.75&0.75&0.75\\
\hline
\hline
$c_{DG}$&0.67&0.67&0.67&0.67&0.67&0.75&0.75&0.75&0.75&0.75\\
\hline
$c_{SD}$&0.67&0.67&0.67&0.67&0.67&0.75&0.75&0.75&0.75&0.75\\
\hline
$c_{HU}$&0.66&0.66&0.66&0.66&0.66&0.75&0.75&0.75&0.75&0.75\\
\hline
$c_{+}$&0.65&0.65&0.65&0.65&0.65&0.75&0.75&0.75&0.75&0.75\\
\hline
$10^{5}$&0.23&0.23&0.23&0.23&0.23&0.38&0.38&0.38&0.38&0.38\\
\hline
\end{tabular}
}
\caption{$s_{theory}^{*}$ compared to $s_{numerical}^{*}$ for the BR2 scheme with $p=2$ and $p=3$ for 32 elements.}
\label{tab:result BR2 Np=3}
\end{table}
\section{Von Neumann analysis}\label{sec:Von neumann analysis}

The objective of this section is to investigate the dissipative and dispersive properties of the ESFR methods while employing the IP/BR2 numerical fluxes through a one-dimensional von Neumann analysis. Similarly to~\cite{vincent_insights_2011} we admit plane wave solutions for a range of frequencies and analyze the dissipative and dispersive properties of the schemes. A von Neumann analysis for the  linear advection can be found in \cite{vincent_insights_2011}.

\subsection{Dissipation and dispersion study}

We start the analysis by first nondimensionalizing the PDE~\eqref{eq:diffusion equation} by introducing $\hat{x}=\frac{x}{\Delta h}$, $\hat{t}=\frac{t\, b}{\Delta h^{2}}$,
\begin{equation}\label{eq:Nondim diffusion equation}
\dfrac{\partial u}{\partial \hat{t}}=\dfrac{\partial^{2}u}{\partial \hat{x}^{2}}.
\end{equation}
Equation~\eqref{eq:diffusion equation} is recovered by taking $b=1$ and a unitary element size $h=1$. 

Upon renaming $\hat{x}$ and $\hat{t}$ into $x$ and $t$, the PDE~\eqref{eq:Nondim diffusion equation} admits solutions of the form $u=e^{\iu\left(kx-\omega t\right)}$, where $k$ is the wave number and $\omega$ is the wave frequency and must satisfy,

\begin{equation}\label{eq:exact dissipation and dispersion relations}
\begin{array}{lll}
\operatorname{Re}(\omega)&=&0,\\
\operatorname{Im}(\omega)&=&-k^{2},
\end{array}
\end{equation}
which are called the exact dissipation and dispersion relations.

We now consider our numerical model (\eqref{eq:numerical main} and~\eqref{eq:numerical aux}) and discretize equation~\eqref{eq:Nondim diffusion equation}. Consider a single element $\Omega_{n}$ of size $\Delta h=h_{n}=1$. The IP scheme, results to a three-element stencil, where the solution in element $\Omega_{n}$ requires information only from $\Omega_{n+1}$ and $\Omega_{n-1}$. A semi-discrete form of the ESFR scheme can be written as, 
\begin{equation}
\dfrac{\mbox{d} \mathbf{\hat u}_{n}}{\mbox{d}t}=\mathbf{C_{n-1}}\mathbf{\hat u}_{n-1}+\mathbf{C_{n}}\mathbf{\hat u}_{n}+\mathbf{C_{n+1}}\mathbf{\hat u}_{n+1},
\end{equation}
where, $\mathbf{C_{n-1}}$ corresponds to a matrix of size  $N_{p}\times N_{p}$ multiplying the vector  $\mathbf{\hat u}_{n-1}$ to compute the Laplacian of $u_{n}$, and the same can be said of matrices $\mathbf{C_{n}}$ and $\mathbf{C_{n+1}}$.  This problem admits a Bloch-wave solution~\cite{AINSWORTH2004106}, $\mathbf{\hat{u}}_{n}=e^{\iu\left(n\,k-\omega^{\delta}t\right)}\mathbf{\hat v}$, where $k$ is the prescribed discrete wave number $\left(-\pi\leq k \leq \pi\right)$, $\omega^{\delta}$ is the discrete frequency number and $\mathbf{\hat v}$ is a vector independent of element $\Omega_{n}$. Both $\omega^{\delta}$ and $\mathbf{\hat v}$ will be determined via the global matrix $\mathbf{S}$ defined below. The solution on the neighbouring elements are calculated using the property of periodicity,
\begin{equation}
\begin{array}{lll}
\mathbf{\hat u}_{n-1}&=&e^{\iu \left((n-1)k -\omega^{\delta} t\right)} \mathbf{\hat v}\\
&=&\mathbf{\hat u}_{n}e^{-\iu k},
\end{array}
\end{equation}
while $\mathbf {\hat u}_{n+1}=\mathbf{\hat u}_{n}e^{\iu k}$ and thus we can represent the final semi-discrete form of the equation as 
\begin{equation}\label{eq:equation S}
\dfrac{\mbox{d} \mathbf{\hat u}_{n}}{\mbox{d}t}=\mathbf{S}\left(k\right)\mathbf{\hat u}_{n},
\end{equation}
where, $\mathbf{S}$ is of size $N_{p}\times N_{p}$, $\mathbf{S}\left(k\right)=\sum\limits_{j=n-1}^{n+1}\mathbf{C}_{j}e^{\iu k\, \left(j-n\right)}$.

The dissipative and dispersive properties of the IP scheme can be studied through the $N_{p}$ eigenvalues $\lambda\left(k\right)$ and $N_{p}$ eigenvectors $\mathbf{\hat v}$ of matrix $\mathbf{S}$, where each eigenvalue corresponds to a specific mode of the solution \cite{vincent_insights_2011}.

\begin{equation}
\begin{array}{lll}
\dfrac{\mbox{d}\mathbf{\hat u}_{n}}{\mbox{d}t}&=&\mathbf{S}\left(k\right)\mathbf{\hat u}_{n},\\
-\iu\omega^{\delta}\mathbf{\hat v}&=&\lambda\left(k\right)\mathbf{\hat v},\\
\omega^{\delta}&=&\iu\lambda.
\end{array}
\end{equation}

Solving this eigenvalue problem results in
\begin{equation}
\mathbf{\hat{u}_{n}}=\mathlarger{\sum}_{j=1}^{N_{☺p}}\alpha_{j}\mathbf{\hat{v}}_{j}e^{\left(\iu k x +\lambda_{j}\left(k\right)t\right)},
\end{equation}
where $\left(\alpha_{j}\right)_{j\in\llbracket 1,N_{p}\rrbracket}$ are the scaling coefficients and can be retrieved from the initial solution \cite{HU1999921}.

Theoretically, we have the dissipation and dispersion relations for $\lambda$ (analogous to relations~\eqref{eq:exact dissipation and dispersion relations}) as,
\begin{equation}\label{eq:dissipation dispersion lambda}
\begin{array}{lll}
\operatorname{Re}\left(\lambda\right)&=&-k^{2},\\
\operatorname{Im}\left(\lambda\right)&=&0.
\end{array}
\end{equation}
In each of our numerical simulations, we have observed no dispersion as expected with the previous relation. Therefore only the dissipation curves will be represented in the following section. The matrix $\mathbf{S}$ is $2\pi$-periodic with $k$; hence the eigenvalues, associated to $\mathbf{S}$, are also $2\pi$-periodic. Having a $p$ degree polynomial, $\mathbf{S}$ is a square matrix of size $p+1$. Each value of $k$ in $\left[-\pi,\pi\right]$ will produce $p+1$ eigenvalues, which provide both the dissipative and dispersive responses for each mode.  Any conclusions drawn from von Neumann analysis is far from trivial. Disparate interpretations of this analysis can be found in the literature. For instance Huynh~\cite{huynh_reconstruction_2009} considers the primary mode to be the physical mode and represents the true dissipative response of the scheme; while the subsequent modes are spurious. However, Van den Abeele~\cite{van2009development} argues that the spurious modes do have an influence on an extended domain of the wave number. Termed as true wave number $\tilde{k}$, the approach plots each mode onto a disjointed domain. However, neither Huynh's~\cite{huynh_reconstruction_2009} classical approach nor Van den Abeele's~\cite{van2009development} approach of investigating the spurious modes consider the effect of the coupling between the modes. In this article, the authors propose an analysis similar to that adopted by Vincent et al.~\cite{vincent_insights_2011} which is based on the approach introduced in~\cite{van2009development} but we attempt to investigate the coupling of the modes. 

To demonstrate the analysis we first consider the dissipation response for $p=1$, where the eigenvalues for the two modes can be found analytically through the expression, 
\begin{equation}
\lambda_{1,2}\left(k\right)=\dfrac{Tr\left(\mathbf{S}\left(k\right)\right)}{2}\pm\dfrac{\sqrt{\Delta\left(k\right)}}{2}.
\end{equation}
Here $Tr\left(S\right)$ is the trace of the matrix $\mathbf{S}$ and $\Delta\left(k\right)=Tr\left(\mathbf{S}\left(k\right)\right)^{2}-4\left|\mathbf{S}\left(k\right)\right|$, while $\left|\mathbf{S}\left(k\right)\right|$ is the determinant of $\mathbf{S}\left(k\right)$. Figure~\ref{fig:dissipation p=1} illustrates the dissipation response for $p=1$ for both $\tau=\tau_{theory}$ and $\tau=4\tau_{theory}$. Here, for both cases, we have $\forall k \in \left[-\pi,\pi\right],\,\Delta\left(k\right)>0$ hence the modes are $\mathcal{C}^{\infty}$. It is possible to obtain cases where $\exists k\in\left[-\pi,\pi\right]\,\text{such that}\, \left|\mathbf{S}\right|=0$. In that case the mode will just be $\mathcal{C}^{0}$.
\begin{figure}[H]
\centering
\begin{subfigure}{0.48\textwidth}
\centering
\includegraphics[width=\linewidth]{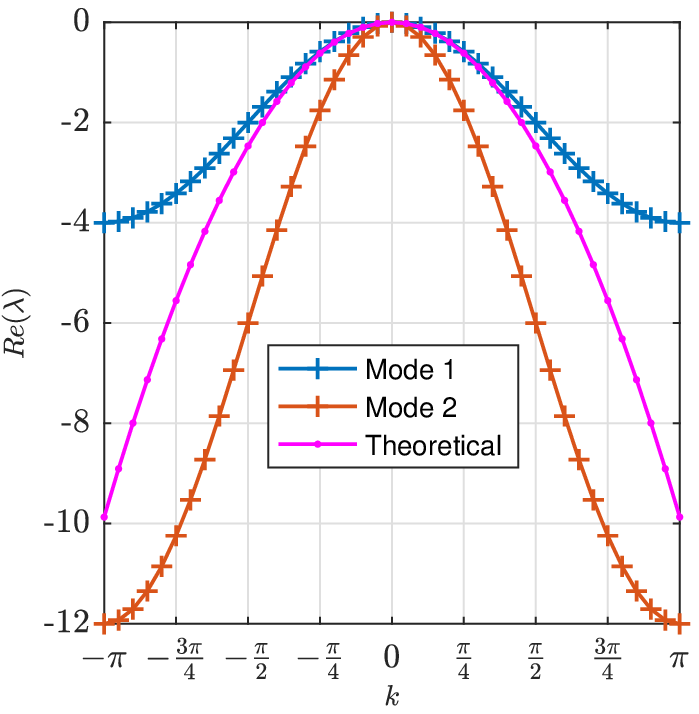}
\caption{Dissipation response $\tau=\tau_{theory}$.} \label{fig:DG tau p=1}
\end{subfigure}\hspace*{\fill}
\begin{subfigure}{0.48\textwidth}
\centering
\includegraphics[width=\linewidth]{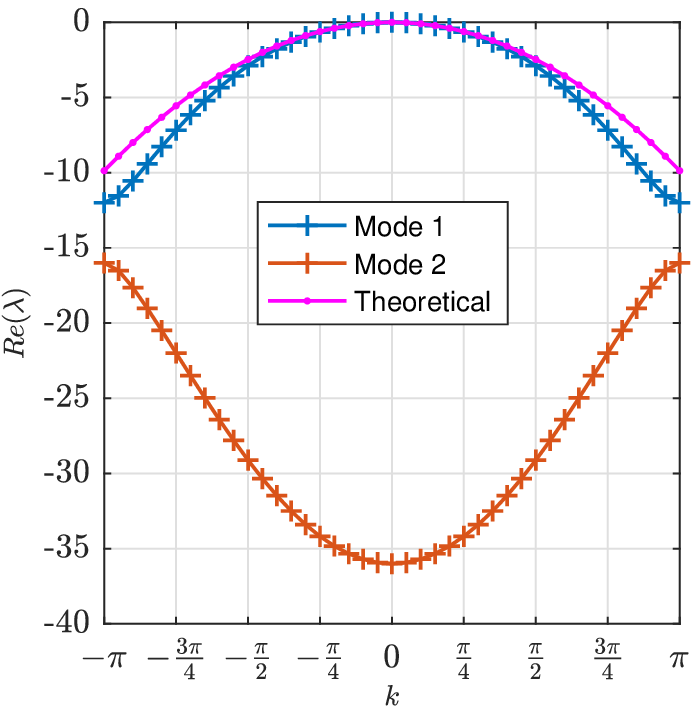}
\caption{Dissipation response $\tau=4\tau_{theory}$.} \label{fig:DG 15 tau p=1}
\end{subfigure}
\caption{Analytical dissipation response for $c_{DG}/\kappa_{DG}$ and $p=1$.} \label{fig:dissipation p=1}
\end{figure}

The case $p=1$ is extremely simple since an analytical expression can be found. However the case $p=2$ is more challenging and we will focus on this degree of polynomial for the rest of the analysis. Once the eigenvalues are evaluated, the primary difficulty resides in the association of the eigenvalue to a mode. Two choices are possible: First, compute the eigenvalues for all wave numbers, $k$ in $\left[-\pi,\pi\right]$, then sort the eigenvalues in such a manner to preserve consistency across the wave number spectrum; Second, employ the algorithm of Vincent et al.~\cite{vincent_insights_2011}, where the eigenvectors are decomposed onto a Legendre basis and the coefficients of the Legendre polynomial for each eigenvector are compared to sort the eigenvalues. In this paper, the second approach has been adopted, as it provides a better physical sense, where, with a higher degree polynomial $p$, we can reach a higher frequency~\cite{van2009development}. Decomposing the eigenvectors into a Legendre polynomial basis yields an insight on which mode is associated to the range of frequencies. In addition we present the results on an extended range, where each eigenvalue is associated to a true wave number $\tilde{k}=k+l\pi$ where $l \in\llbracket 0,p\rrbracket $. The algorithm can be summarized as follows, more details are available in~\cite{vincent_insights_2011}, 
\begin{enumerate}
\item Select a value for wave number, $k$.
\item Calculate the eigenvalues $\left(\lambda_{j}\right)_{j\in\llbracket 1,N_{p} \rrbracket}$ and the associated eigenvectors $\left(\mathbf{\hat v}_{j}\right)_{j\in\llbracket 1,N_{p} \rrbracket}$.
\item Decompose each eigenvector into its Legendre polynomial basis. $\mathbf{\hat v}_{j}$ yields $N_{p}$ discrete values from which we can define a polynomial of degree $p$: $\hat{v}_{j}\left(r\right)$. This function can be represented as a sum of Legendre polynomials $\hat{v}_{j}\left(r\right)=\sum_{i=1}^{N_{p}}{\bar{v}}_{i,j}\Psi_{i-1}\left(r\right)$, with $\Psi_{i}$, the Legendre polynomial of degree $i$. Therefore the vector $\mathbf{\hat{v}}_{j}$ is associated to the vector $\mathbf{\bar{v}}_{j}$, which contains the values $\bar{v}_{i,j}$. Then build the matrix $\mathbf{\bar{V}}\left(k\right)$ where ${\bar{v}}_{i,j}$ is the coefficient of the Legendre polynomial $P_{i}$ associated to the eigenvalues $\lambda_{j}$.
\item Starting with the first row of $\mathbf{\bar{V}}$, ie. $\Psi_{0}$, compare the values of this first row and associate Mode 1 to the eigenvalue whose coefficient ${\bar{v}}_{1,j}$ is the greatest. Then remove the associated column to the matrix $\mathbf{\bar{V}}$ since the corresponding eigenvalue has been assigned a mode and then move on to the second row of $\mathbf{\bar{V}}$.
For instance the eigenvalue which has the highest coeffcient for $\Psi_{0}$ is associated to low frequencies and hence to Mode 1, whereas the eigenvalues which has the highest coefficient for $\Psi_{p}$ is associated to high frequencies and hence to the last mode.
\end{enumerate}

For a better comprehension, let us take for example the case of $p=2$. The decomposition into a Legendre polynomial basis leads us to the matrix,
\begin{equation}
\mathbf{\bar{V}} =
\begin{bmatrix}
{\bar{v}}_{11}&{\bar{v}}_{12}&{\bar{v}}_{13}\\
{\bar{v}}_{21}&{\bar{v}}_{22}&{\bar{v}}_{23}\\
{\bar{v}}_{31}&{\bar{v}}_{32}&{\bar{v}}_{33}
\end{bmatrix}.
\end{equation}
Let us imagine that the $\max\left({\bar{v}}_{11},{\bar{v}}_{12},{\bar{v}}_{13}\right)={\bar{v}}_{13}$ then $\lambda_{3}$ is associated to Mode 1. The third column is removed from consideration when the next row is analyzed. Then, if $\max\left({\bar{v}}_{21},{\bar{v}}_{22}\right)={\bar{v}}_{21}$, the first eigenvalue computed, $\lambda_{1}$ is associated to Mode 2. It remains that $\lambda_{2}$ corresponds to Mode 3.

From our previous example $p=1$, the second approach as described above yields the dissipation response demonstrated in Figures~\ref{fig:DG tau p=1 legendre} and~\ref{fig:DG 15 tau p=1 legendre}. While Figure~\ref{fig:DG tau p=1 legendre} is identical to Figure~\ref{fig:DG tau p=1} for $\tau=\tau_{theory}$, we observe that Figure~\ref{fig:DG 15 tau p=1 legendre} is largely similar to~Figure~\ref{fig:DG 15 tau p=1} when $\tau=4\tau_{theory}$ but the modes are shifted for $|k|$ close to $\pi$, where the procedure has a tendency to introduce jumps within a mode if the maximum coefficients of the Legendre polynomial shifts between two eigenvalues when a particular degree of the polynomial is considered.

\begin{figure}[H]
\centering
\begin{subfigure}{0.48\textwidth}
\centering
\includegraphics[width=\linewidth]{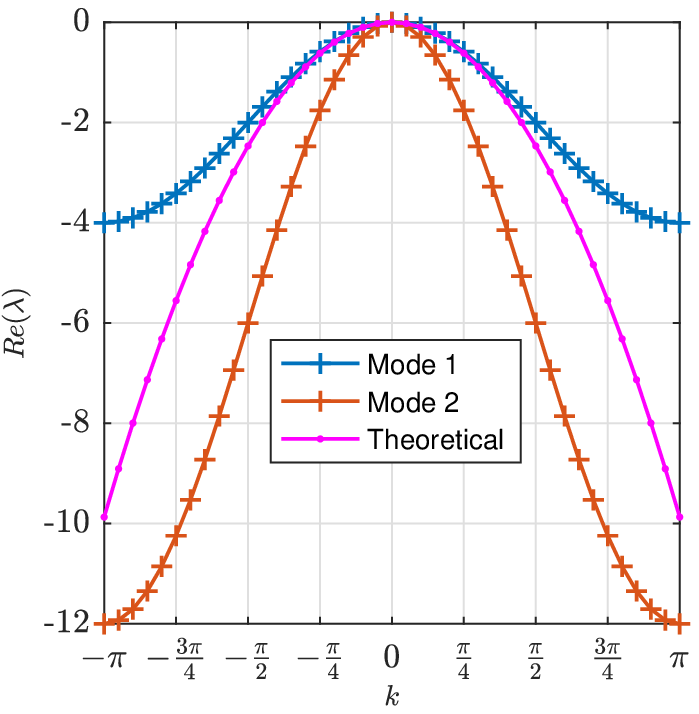}
\caption{Dissipation response $\tau=\tau_{theory}$.} \label{fig:DG tau p=1 legendre}
\end{subfigure}\hspace*{\fill}
\begin{subfigure}{0.48\textwidth}
\centering
\includegraphics[width=\linewidth]{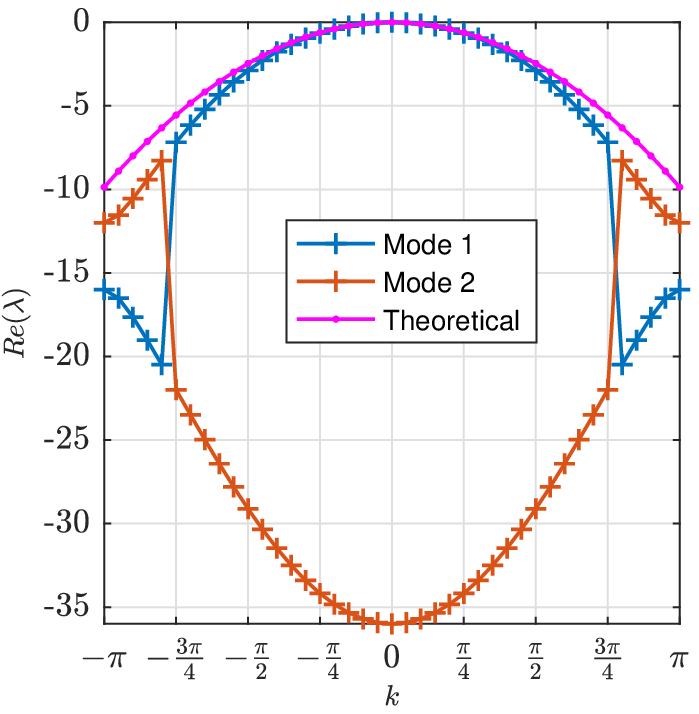}
\caption{Dissipation response $\tau=4\tau_{theory}$.} \label{fig:DG 15 tau p=1 legendre}
\end{subfigure}

\caption{Dissipation response for $c_{DG}/\kappa_{DG}$ and $p=1$ using Legendre decomposition.} \label{fig:dissipation p=1 Legendre}
\end{figure}

As demonstrated in Figures~\ref{fig:DG tau p=1 legendre} and~\ref{fig:DG 15 tau p=1 legendre}, this decomposition leads to two challenges. First, by removing the eigenvalues that have already been affected, we remove the influence of their coefficients associated to higher degree polynomials. In the previous example, we didn't compare ${\bar{v}}_{2,3}$ or ${\bar{v}}_{3,3}$ with the other values. Therefore $\lambda_{3}$, associated to Mode 1, may have some contribution to Mode 2 or Mode 3 that was not taken into account. To better quantify the potential that an eigenvalue can influence a wider range of frequencies, we can examine the ratio of the primary component of the Legendre polynomial, ${\bar{v}}_{ij}$, for Mode $i$ to the sum of its coefficients. 
Let us consider a Mode $i$ associated to the vector $\mathbf{\bar{v}}_{j}$, where the \enquote{primary} component of this vector, ${\bar{v}}_{i,j}$ is aligned with Mode $i$, then the ratio $R_{mode}(i,k)$ is defined as
\begin{equation}\label{eq:ratio mode}
R_{mode}\left(i,k\right)=\dfrac{{\bar{v}}_{i,j}^{2}}{\sum_{m=1}^{N_{p}}{\bar{v}}_{m,j}^{2}},
\end{equation}
where $R_{mode}\left(i,k\right)$ is the ratio of Mode $i$ for a wave number $k$. Therefore, this ratio determines whether considering just one component of the projected eigenvector is not too restrictive. If the ratio is close to unity then it signifies that the eigenvalue considered is indeed just associated to the mode. Conversely if the ratio is less than 1 then it means that the eigenvalue influences a wider range of frequencies.

Second, by comparing the maximum of the coefficients, we acknowledge that there are contributions from the other eigenvalues to the mode considered. Assigning each eigenvalue to a single mode leads to an erroneous decoupling of the modes. To represent the degree to which the modes are decoupled, we examine the ratio of ${\bar{v}}_{ij}$ for Mode $i$ to the sum of the coefficients $i$ across all eigenvectors. Thus we define $R_{energy}(i,k)$ as,
\begin{equation}\label{eq:ratio energy}
R_{energy}\left(i,k\right)=\dfrac{{\bar{v}}_{i,j}^{2}}{\sum_{m=1}^{N_{p}}{\bar{v}}_{i,m}^{2}},
\end{equation}
where $R_{energy}\left(i,k\right)$ is the ratio of Mode $i$ for a wave number $k$. If this ratio is close to 1, then the other modes do not affect the current one. Although both ratios enable us to see how the modes are coupled, they provide us with different insights: while $R_{mode}$ indicates the amount by which a mode is restrainted to its domain of frequency, $R_{energy}$ provides how much information is lost by considering one eigenvalue to one mode. If both ratios are equal to unity then the mode considered has an influence on its own domain of frequency and the other modes have no influence over it. Since Figure~\ref{fig:dissipation p=1 Legendre} was obtained via Legendre decomposition, we can in fact represent it in the extended range.

\begin{figure}[H]
\centering
\includegraphics[width=15cm]{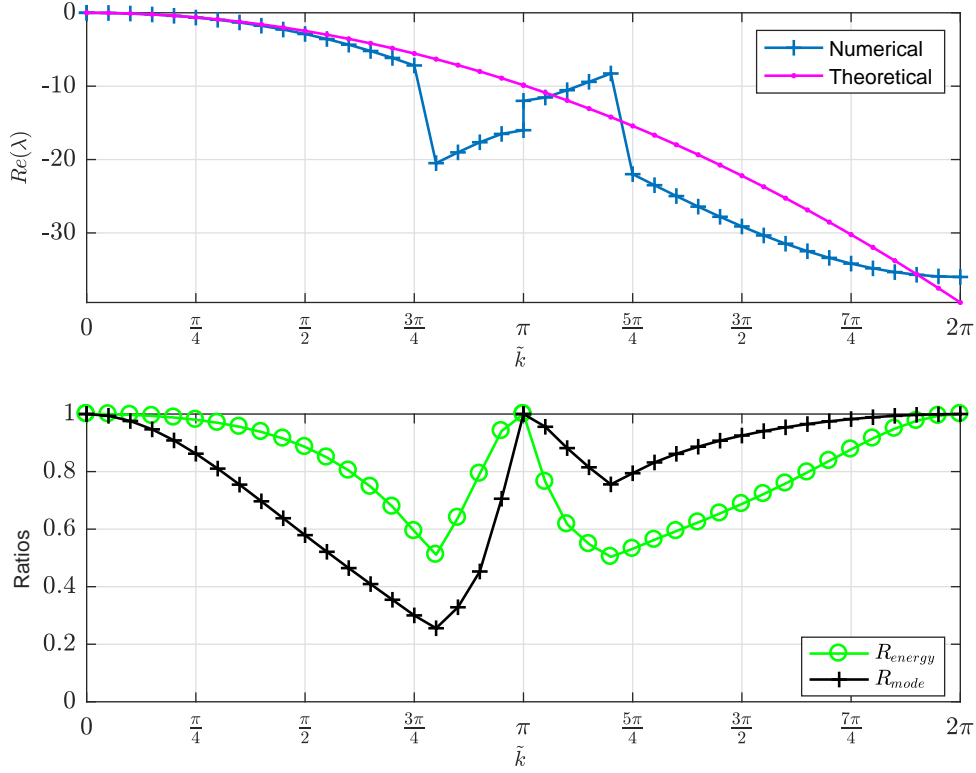}
\caption{Dissipation response along $\tilde{k}$ for $c_{DG}/\kappa_{DG}$, $\tau=4\tau_{theory}$ and $p=1$.}
\label{fig:dissipation p=1 extended}
\end{figure}

To the authors' opinion, Figure~\ref{fig:dissipation p=1 extended} provides, in the simplest way, all the information that we can recover from von Neumann analysis. By projecting the eigenvectors in Legendre polynomial basis, we obtain the modes with a consistent algorithm. The uncoupled dissipative response of the scheme is compared against the theory when the modes are presented in the extended domain. Lastly, by representing the ratios, we enable the readers to observe that for most of the extended domain the modes are coupled. Note that the discontinuities in Figure~\ref{fig:dissipation p=1 extended} are of two types, where the jump at $\pi$ is due to the disparate values of the real part of the eigenvalue between modes 1 and 2, while at the jumps at $\frac{3\pi}{4}$ and $\frac{5\pi}{4}$ are due to the shifts of the maximum coefficients of the Legendre polynomial between eigenvalues.

In the following analysis, we investigate the influence of $\tau$ and $c$. The comparison will bear on the dissipation response for decoupled modes (enlightened by the representation of the ratios). Figure~\ref{fig:Dissipation_extended_DG_tau_th} is our reference graph, where, for high frequencies, the scheme is less dissipative than the theory. We observe five different jumps across the range of frequencies, the jumps at $\tilde{k}=\frac{3\pi}{4}$ and $\tilde{k}=\frac{11\pi}{4}$ are due to shifts between modes 1 and 3. The same can be said for the jumps at $\tilde{k}=\frac{3\pi}{2}$ and $\tilde{k}=\frac{5\pi}{2}$, where the shifts are between modes 2 and 3. These shifts occur when the value of the coefficients $\bar{v}_{ij}$, between modes are very similar, which causes a selection of a different eigenvalue to represent a particular mode. The fifth and final jump or discontinuity occurs at $2\pi$ and is due to the difference in dissipation response between the modes at the $k=\pi$ frequency. By increasing $\tau$, we retrieve the results in Figure~\ref{fig:Dissipation_DG_tau_th_1_5}. For low $\tilde{k}$, the response is unaffected by the penalty term. However on the domain $\left[\frac{\pi}{2}, 3\pi\right]$, the response is more dissipative compared to the lower $\tau$ case as shown in Figure~\ref{fig:Dissipation_extended_DG_tau_th}. We also observe a smoother dissipation curve with two discontinuities for $\tilde{k}$ close to $\pi$ and $2\pi$, where mode 1 represents the region in the range $[0,\pi]$, while $[\pi,2\pi]$ is reserved for mode 2 and the last interval for mode 3. Unlike the jumps noted above these discontinuities are due to disparate values of the dissipation response at the $k=\pi$ frequency. For $\tilde{k}$ close to $2\pi$, both ratios are close to 1. We see that mode 2 is a bit over-dissipative. Therefore as expected, raising the value of $\tau$ increases the dissipation; the ratios of $R_{energy}\left(i,k\right)$ and $R_{mode}\left(i,k\right)$ have a similar shape, compared to Figure~\ref{fig:Dissipation_extended_DG_tau_th}, in the range of $[0,\frac{5\pi}{4}]$, but change in a non-monotonic fashion in the rest of the extended domain with a general trend towards a greater coupling of the modes.
 
We now investigate the impact of parameter $c$. Numerically, increasing $c$ leads to dampening the highest degree of polynomial in the test function~\cite{zwanenburg_equivalence_2016}. Referring to Figure~\ref{fig:Dissipation_c+_kappa+_tau_th}, we observe a slight jump for $\tilde{k}\sim\frac{\pi}{2}$. For higher $\tilde{k}$, the response becomes less dissipative compared to Figure~\ref{fig:Dissipation_extended_DG_tau_th}. We obtained a similar result for $p=1$ where the eigenvalues were computed analytically. The ratios of $R_{energy}\left(i,k\right)$ and $R_{mode}\left(i,k\right)$ remains mostly unperturbed at lower frequencies, $[0,\pi]$ but we obtain a reduction of the coupling at higher frequencies. 

In conclusion, we can make the following general observations of the impact the parameters $\tau$ and $c$ have on the dissipative properties of the scheme. First, for all considered values of these parameters, we verify $\operatorname{Re}(\lambda)\leq 0$ and the scheme remains stable. Second, as $\tau$ increases, the Bloch wave dissipation relation demonstrates that the scheme becomes more dissipative. Lastly, raising $c$ results in less dissipation.

\begin{figure}[H]
\centering
\includegraphics[width=14cm]{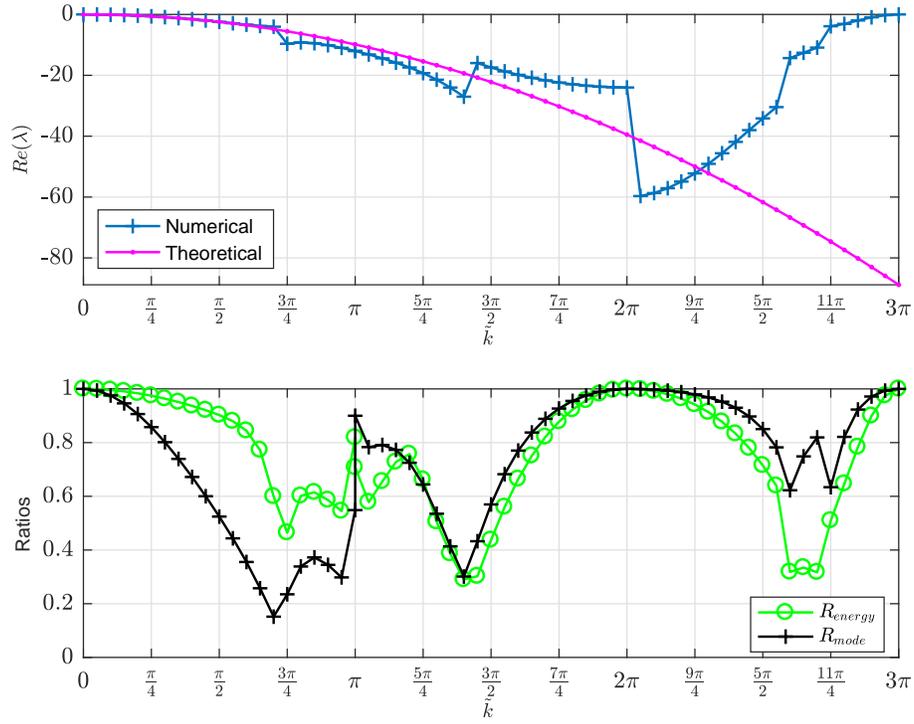}
\caption{Dissipation response along $\tilde{k}$ for $c_{DG}/\kappa_{DG}$, $\tau=\tau_{theory}$ and $p=2$.}
\label{fig:Dissipation_extended_DG_tau_th}
\end{figure}

\begin{figure}[H]
\centering
\includegraphics[width=14cm]{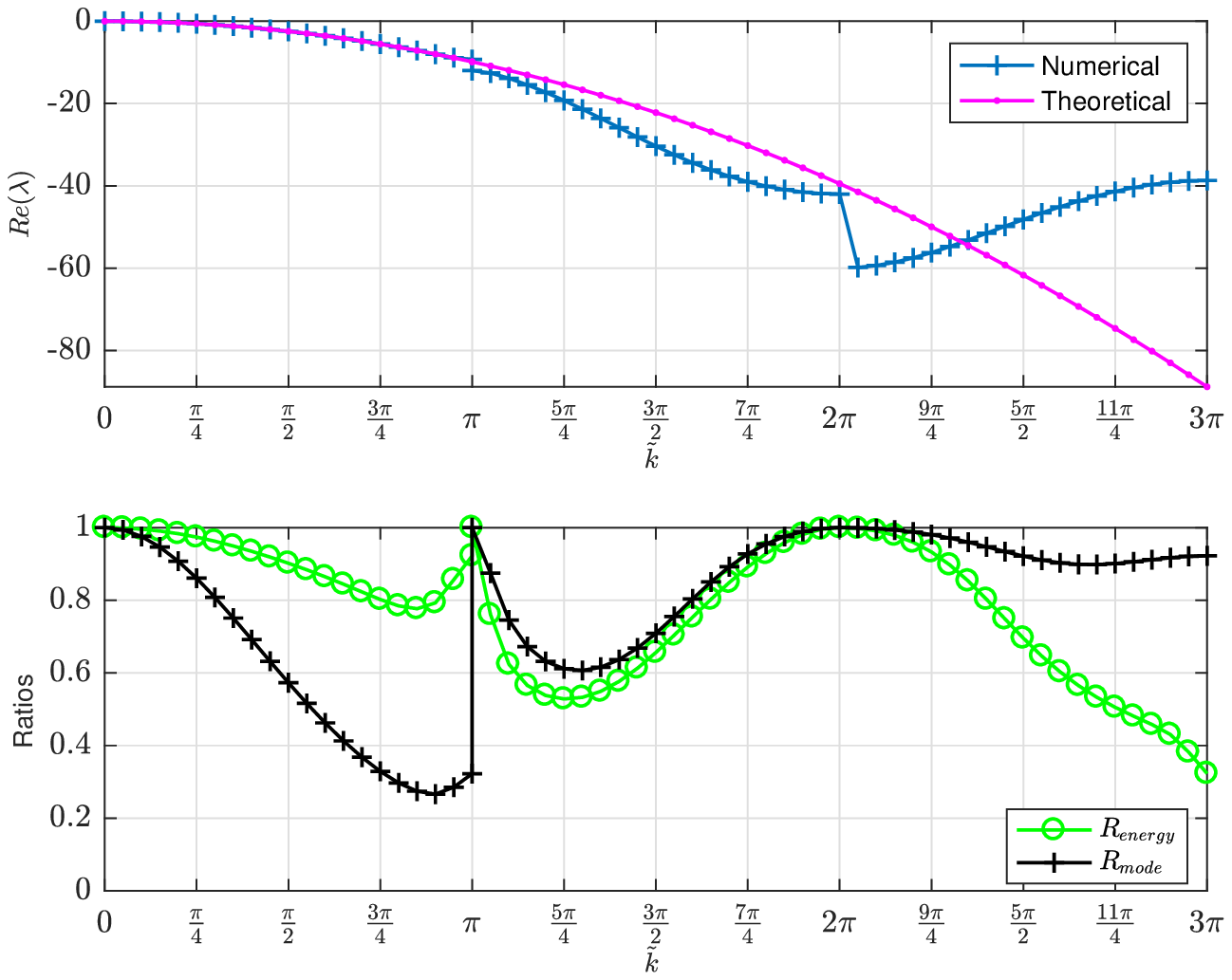}
\caption{Dissipation response along $\tilde{k}$ for $c_{DG}/\kappa_{DG}$, $\tau=1.5\tau_{theory}$ and $p=2$.}
\label{fig:Dissipation_DG_tau_th_1_5}
\end{figure}

\begin{figure}[H]
\centering
\includegraphics[width=14cm]{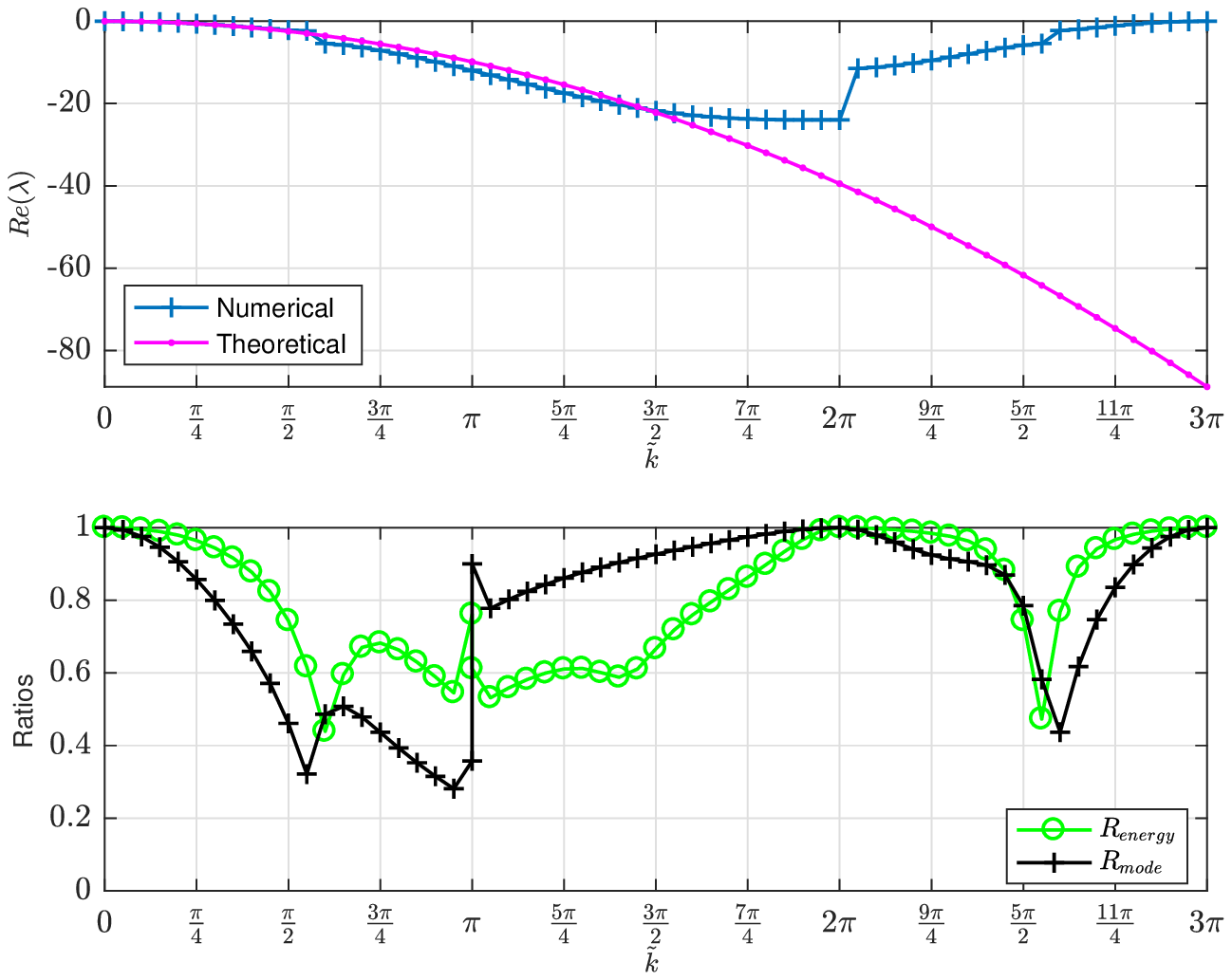}
\caption{Dissipation response along $\tilde{k}$ for $c_{+}/\kappa_{DG}$, $\tau=\tau_{theory}$ and $p=2$.}
\label{fig:Dissipation_c+_kappa+_tau_th}
\end{figure}

\subsection{Maximal time steps}
As explained in Section~\ref{sec:IP numerical experiment} we use the RK54 scheme as our time solver, through~\eqref{eq:equation S}, and obtain,
\begin{equation}\label{eq:Temporal scheme}
\mathbf{\hat{u}}_{n}^{t_{n+1}}=\mathbf{M}\left(k,\Delta t\right)\mathbf{\hat{u}}_{n}^{t_{n}},
\end{equation}
where, $\mathbf{M}$ is defined as,
\begin{equation}
\mathbf{M}\left(k,\Delta t\right)=1+\Delta t\,\mathbf{S}\left(k\right)+\dfrac{1}{2!}\left(\Delta t\,\mathbf{S}\left(k\right)\right)^{2}+\dfrac{1}{3!}\left(\Delta t\,\mathbf{S}\left(k\right)\right)^{3}+\dfrac{1}{4!}\left(\Delta t\,\mathbf{S}\left(k\right)\right)^{4}+\dfrac{1}{200}\left(\Delta t\,\mathbf{S}\left(k\right)\right)^{5},
\end{equation}
$\mathbf{\hat{u}}_{n}^{t_{n+1}}$ represents the column vector solution on the element $\Omega_{n}$ at time $t_{n+1}$, $\mathbf{\hat{u}}_{n}^{t_{n}}$ denotes the column vector solution on the element $\Omega_{n}$ at time $t_{n}$ and  $\Delta t$ represents the time step.

The matrix $\mathbf{M}$ depends on both the wave number $k$ and the time step $\Delta t$, where $k$ varies between $\left[0,2\pi\right]$ and any value can be chosen for $\Delta t$. Then the matrix $\mathbf{M}$ is decomposed in terms of its eigenvalues ($\lambda_{\mathbf{M}}$) and eigenvectors. To ensure stability of the scheme~\eqref{eq:Temporal scheme}, the modulus of the spectral radius of $\mathbf{M}$ must be less than 1. Starting at an initial $\Delta t_{0}$ sufficiently high to produce an unstable solution, we scan over the range of $k$ and compute the largest eigenvalues of $\mathbf{M}$: $\lambda_{\mathbf{M},max}$, then decrease $\Delta t$ as long as $\lambda_{\mathbf{M},max}\geq 1$. Applying this procedure, we obtain the maximal time step, $\Delta t_{max}$ for combinations of the parameters $c, \kappa$, and $\tau$. Four values of $c$: $c_{DG}$, $c_{SD}$, $c_{HU}$, and $c_{+}$ will be studied. For each value of $c$, two values for $\kappa$ ($\kappa_{DG}$ and $\kappa_{+}$) are considered.  Finally, three values of $\tau$ ($\tau_{theory}$, $1.1\tau_{theory}$ and $1.5\tau_{theory}$) are used for each combination of $c$ and $\kappa$ to better understand the impact of the value that controls the jump in the solution when evaluating the common viscous numerical flux. 

\begin{table}[H]
\resizebox{\textwidth}{!}{
\begin{tabular}{|c|c||c|c|c||c|c|c|}
\hline
\multicolumn{2}{|c||}{$p$} &\multicolumn{3}{c||}{2} &\multicolumn{3}{c|}{3}\\
\hline
$c$ &$\kappa$ &$\tau_{theory}^{*}$&1.1$\tau_{theory}^{*}$ &1.5$\tau_{theory}^{*}$ &$\tau_{theory}^{*}$ &1.1$\tau_{theory}^{*}$ &1.5$\tau_{theory}^{*}$ \\ 
\hline
\multirow{2}{*}{$c_{DG}$} &$\kappa_{DG}$&7.76e-02 &7.76e-02 &7.76e-02 &2.74e-02 &2.74e-02 &2.74e-02 \\
\cline{2-8}
&$\kappa_{+}$&7.76e-02 &7.76e-02 &7.76e-02 &2.74e-02 &2.74e-02 &2.74e-02 \\
\hline
\multirow{2}{*}{$c_{SD}$} &$\kappa_{DG}$&1.29e-01 &1.29e-01 &1.11e-01 &4.74e-02 &4.74e-02 &3.19e-02 \\
\cline{2-8}
&$\kappa_{+}$&1.29e-01 &1.29e-01 &1.11e-01 &4.74e-02 &4.74e-02 &3.19e-02 \\
\hline
\multirow{2}{*}{$c_{HU}$} &$\kappa_{DG}$&1.82e-01 &1.68e-01 &1.11e-01 &5.62e-02 &5.15e-02 &3.19e-02 \\
\cline{2-8}
&$\kappa_{+}$&1.82e-01 &1.68e-01 &1.11e-01 &5.62e-02 &5.15e-02 &3.19e-02 \\
\hline
\multirow{2}{*}{$c_{+}$} &$\kappa_{DG}$&1.94e-01 &1.69e-01 &1.11e-01 &5.99e-02 &5.25e-02 &3.19e-02 \\
\cline{2-8}
&$\kappa_{+}$&1.94e-01 &1.69e-01 &1.11e-01 &5.99e-02 &5.25e-02 &3.19e-02 \\
\hline
\end{tabular}
}
\caption{Maximal time step for the IP scheme.}
\label{tab:Maximal time step}
\end{table}

Table~\ref{tab:Maximal time step} presents the maximum time step, $\Delta t_{max}$, for the various combinations of $c, \kappa,$ and $\tau$ at polynomial orders $p=2$ and $3$. From Table~\ref{tab:Maximal time step}, we present three primary observations. First, as expected, larger values of $c$ lead to larger time steps, where $\Delta t_{max}$ is achieved at $c_{+}/\kappa_{+}$. By increasing the parameter $c$, the highest mode is dampened, and thus instabilities provided by a large $\Delta t$ are dampened. In fact for the method $c_{+}$ we have a $\Delta t_{max}$ which is 2.5 times larger than $\Delta t_{max}$ for $c_{DG}$ with $p=2$ and 2.2 times for $p=3$. Second,  increasing $\tau$ lowers the $\Delta t_{max}$. This is rather counter-intuitive since increasing $\tau$ introduces more dissipation and hence instabilities coming from larger $\Delta t$ are expected to be compensated. Third, as expected from section~\ref{sec:independent kappa} $\kappa$ has no influence. We may also notice that the lower the value of parameter $c$, the lesser the influence of $\tau$ on the maximal time step.

\begin{figure}[H] 
\centering

\begin{subfigure}{0.49\textwidth}
\centering
\includegraphics[width=\linewidth]{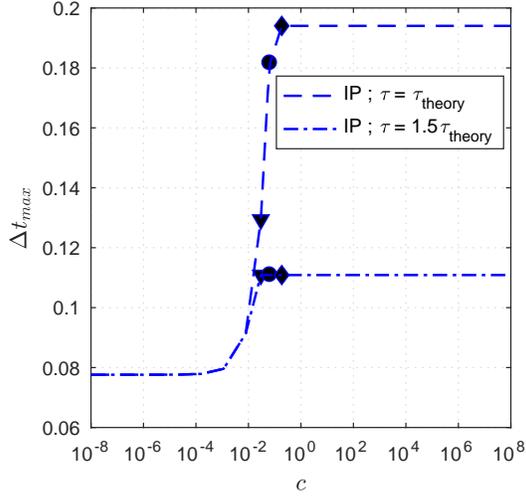}
\caption{$\Delta t_{max}$ for $p=2$.} \label{fig:dtmax DG p2 c}
\end{subfigure}\hspace*{\fill}
\begin{subfigure}{0.49\textwidth}
\centering
\includegraphics[width=\linewidth]{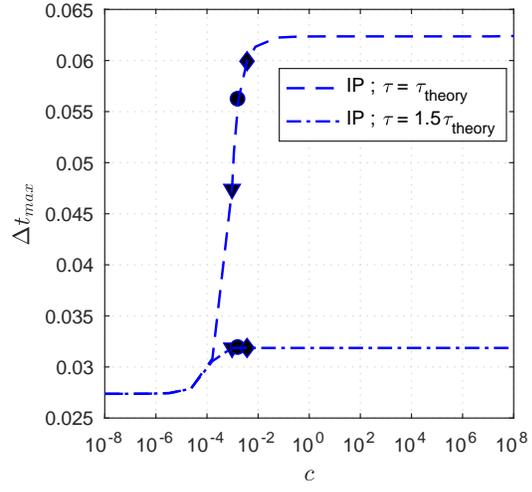}
\caption{$\Delta t_{max}$ for $p=3$.} \label{fig:dtmax DG p3 c}
\end{subfigure}
\caption{$\Delta t_{max}$ along $c$ for different $p$, the lower triangle is for $c=c_{SD}$ the circle is for $c=c_{HU}$ and the diamond for $c=c_{+}$; using log scale the DG case couldn't be represented.}
\label{fig:plot dt max c}
\end{figure}

\begin{figure}[H] 
\centering

\begin{subfigure}{0.49\textwidth}
\centering
\includegraphics[width=\linewidth]{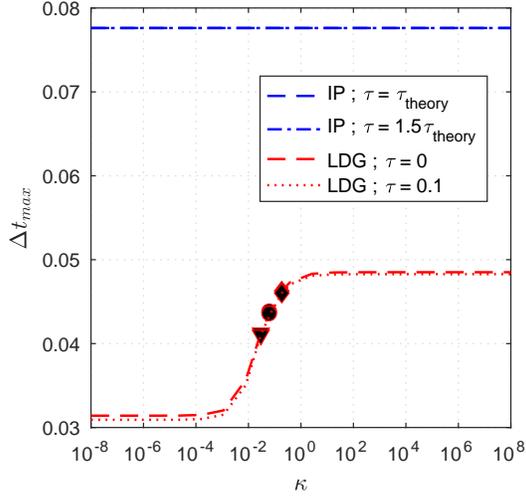}
\caption{$\Delta t_{max}$ for $c=c_{DG}$ $p=2$.} \label{fig:dtmax DG p2}
\end{subfigure}\hspace*{\fill}
\begin{subfigure}{0.49\textwidth}
\centering
\includegraphics[width=\linewidth]{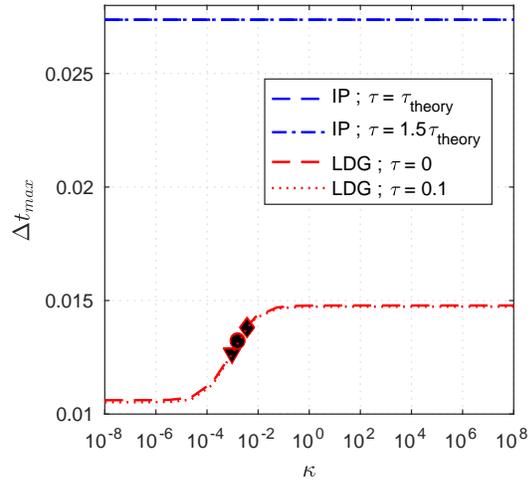}
\caption{$\Delta t_{max}$ for $c=c_{DG}$ $p=3$.} \label{fig:dtmax DG p3}
\end{subfigure}

\medskip
\begin{subfigure}{0.49\textwidth}
\centering
\includegraphics[width=\linewidth]{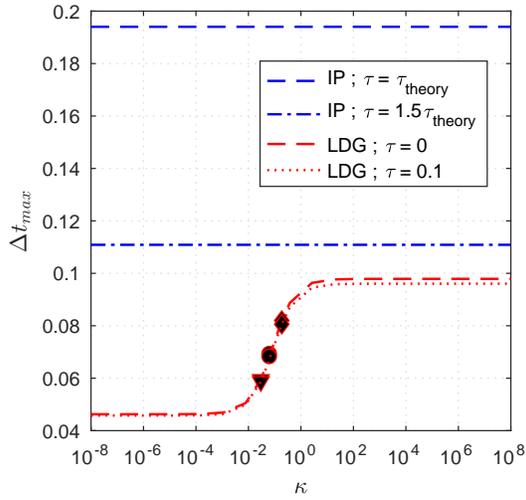}
\caption{$\Delta t_{max}$ for $c=c_{+}$ $p=2$.} \label{fig:dtmax c+ p2 cited}
\end{subfigure}\hspace*{\fill}
\begin{subfigure}{0.49\textwidth}
\centering
\includegraphics[width=\linewidth]{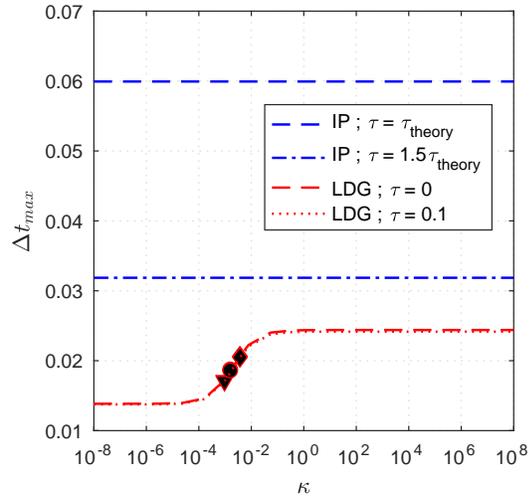}
\caption{$\Delta t_{max}$ for $c=c_{+}$ $p=3$.} \label{fig:dtmax c+ p3 cited}
\end{subfigure}
\caption{$\Delta t_{max}$ along $\kappa$ for different $c$ and $p$, the lower triangle is for $\kappa=\kappa_{SD}$ the circle is for $\kappa=\kappa_{HU}$ and the diamond for $\kappa=\kappa_{+}$; using log scale the DG case couldn't be represented.}
\label{fig:plot dt max}
\end{figure}

Figure~\ref{fig:plot dt max c} studies the influence of both parameters $c$ and $\tau$ on the maximal time-step. As expected increasing $c$ leads to higher time steps. However for $c\geq c_{+}$, $\Delta t_{max}$ reaches a threshold for both values of $p$. For most of the range of $c$, increasing $\tau$ leads to a reduction in the maximal time step. However for very low values of $c$ (such as $c_{DG}$), this effect is not observed.

Figure~\ref{fig:plot dt max} represents the numerical values of the IP scheme (Table~\ref{tab:Maximal time step}) along with the values for the LDG scheme (see~\eqref{eq:numerical fluxes} for the numerical flux, here we chose $\beta=\frac{1}{2}$). For better clarity, only two values of $\tau$ are represented for both schemes. For both the LDG and the IP numerical fluxes, increasing $\tau$ leads to a reduction in the maximal time step; however, we observe that this effect is much more dominant for the IP scheme (Figure~\ref{fig:dtmax c+ p2 cited} and~\ref{fig:dtmax c+ p3 cited}). Therefore, it is preferable to use the lowest value of $\tau$ possible. We notice that the IP scheme has a higher performance in terms of maximal time step compared to the LDG scheme.

\section{$L_{2}$ errors and orders of accuracy}\label{sec:L2 errors}
In this section, investigation on the accuracy and explicit time step limits of the ESFR schemes for the diffusion equation are presented. In particular, we compare the IP/BR2 schemes on a one-dimensional domain compared to the LDG scheme $\left(\beta=\frac{1}{2}\right)$ through a careful study of the impact of the parameters, $\tau$ and $c$.  The problem considered is the same as in Section~\ref{sec:IP numerical experiment}. We employ LGL nodes and the time step has been chosen such that it corresponds to $\Delta t_{max}$. Time advancement was achieved through the RK54 method.

The $L_{2}$-error was computed using LGL quadrature with order $p$ (which provides sufficient strength),
\begin{equation}
\displaystyle{
L_{2}-error=\sqrt{\sum_{n=1}^{N_{K}}\int_{x_{n}}^{x_{n+1}}\left(u_{n}-u_{exact}\right)^{2} \mbox{d}x}
=\sqrt{\sum_{n=1}^{N_{K}}\left(\mathbf{\hat u}^T_{n}-\mathbf{\hat u}^T_{exact}\right)\mathbf{W}\left(\mathbf{\hat u}_{n}-\mathbf{\hat u}_{exact}\right)},
}
\end{equation}
where $\mathbf{W}$ is the quadrature matrix of size $N_{p}\times N_{p}$. The following results may be extensive but the authors prefer to provide the exact numbers, to aid the research community to replicate the results. Both parameters $c$ and $\kappa$ are modified, ranging from $c_{DG}$ to $c_{+}$ and for $\kappa_{DG}$ and $\kappa_{+}$. For brevity, only the results of three meshes are provided: $N_{K} \in \left\lbrace 32,64,128 \right\rbrace$. We also provide the order of accuracy which is calculated based on the adjacent mesh size, as well as the maximal time step for 32 elements. Maximal time steps were evaluated through an iterative approach while ensuring that the solution remains bounded at $t=1$. The numerical simulations have been replicated for two values of $\tau$: $\tau_{theory}$ provided in~\eqref{eq:criterion IP} and $1.5\tau_{theory}$, for both $p=2$ and $p=3$.  
\begin{table}[H]
\resizebox{\textwidth}{!}{
\begin{tabular}{|c|c||c|c|c|ccc|c||c|c|c|ccc|c|}
\hline
& &\multicolumn{7}{c||}{$\tau_{theory}$} &\multicolumn{7}{c|}{$1.5\tau_{theory}$} \\
\hline
$c$ &$\kappa$ &$N_{K}$=32 &$N_{K}$=64 &$N_{K}$=128 &\multicolumn{3}{c|}{OOA} &$\Delta t_{max}$ &$N_{K}$=32 &$N_{K}=64$ &$N_{K}$=128 &\multicolumn{3}{c|}{OOA} &$\Delta t_{max}$ \\
\hline
\multirow{2}{*}{$c_{DG}$} &$\kappa_{DG}$&1.54e-04 &1.92e-05 &2.40e-06 &-&3.01 &3.00 &3.00e-03 &9.24e-05 &1.16e-05 &1.44e-06 &-&2.99 &3.02 &3.00e-03 \\
\cline{2-16}
&$\kappa_{+}$&1.54e-04 &1.92e-05 &2.40e-06 &-&3.01 &3.00 &3.00e-03 &9.24e-05 &1.16e-05 &1.44e-06 &-&2.99 &3.02 &3.00e-03 \\
\hline
\multirow{2}{*}{$c_{SD}$}&$\kappa_{DG}$&1.46e-04 &1.82e-05 &2.27e-06 &-&3.00 &3.00 &5.00e-03 &9.95e-05 &1.21e-05 &1.56e-06 &-&3.04 &2.96 &4.20e-03 \\
\cline{2-16}
&$\kappa_{+}$&1.46e-04 &1.82e-05 &2.27e-06 &-&3.00 &3.00 &5.00e-03 &9.95e-05 &1.21e-05 &1.56e-06 &-&3.04 &2.96 &4.20e-03 \\
\hline
\multirow{2}{*}{$c_{HU}$}&$\kappa_{DG}$&1.35e-04 &1.68e-05 &2.09e-06 &-&3.01 &3.00 &7.00e-03 &8.48e-05 &1.03e-05 &1.32e-06 &-&3.04 &2.96 &4.20e-03 \\
\cline{2-16}
&$\kappa_{+}$&1.35e-04 &1.68e-05 &2.09e-06 &-&3.01 &3.00 &7.00e-03 &8.48e-05 &1.03e-05 &1.32e-06 &-&3.04 &2.96 &4.20e-03 \\
\hline
\multirow{2}{*}{$c_{+}$}&$\kappa_{DG}$&1.14e-04 &1.44e-05 &1.79e-06 &-&2.98 &3.01 &7.50e-03 &6.61e-05 &7.94e-06 &9.99e-07 &-&3.06 &2.99 &4.20e-03 \\
\cline{2-16}
&$\kappa_{+}$&1.14e-04 &1.44e-05 &1.79e-06 &-&2.98 &3.01 &7.50e-03 &6.61e-05 &7.94e-06 &9.99e-07 &-&3.06 &2.99 &4.20e-03 \\
\hline
\end{tabular}
}
\caption{$L_{2}$-errors $p=2$.}
\label{tab:L2 Np3}
\end{table}

\begin{table}[H]
\resizebox{\textwidth}{!}{
\begin{tabular}{|c|c||c|c|c|ccc|c||c|c|c|ccc|c|}
\hline
& &\multicolumn{7}{l||}{$\tau_{theory}$} &\multicolumn{7}{l|}{$1.5\tau_{theory}$} \\
\hline
$c$ &$\kappa$ &$N_{K}$=32 &$N_{K}$=64 &$N_{K}$=128 &\multicolumn{3}{c|}{OOA} &$\Delta t_{max}$ &$K$=32 &$K=64$ &$K$=128 &\multicolumn{3}{c|}{OOA} &$\Delta t_{max}$ \\
\hline
\multirow{2}{*}{$c_{DG}$} &$\kappa_{DG}$&1.60e-05 &2.00e-06 &2.49e-07 &-&3.01 &3.00 &1.00e-03 &1.51e-06 &9.79e-08 &6.13e-09 &-&3.94 &4.00 &1.00e-03 \\
\cline{2-16}
&$\kappa_{+}$&1.60e-05 &2.00e-06 &2.49e-07 &-&3.01 &3.00 &1.00e-03 &1.51e-06 &9.79e-08 &6.13e-09 &-&3.94 &4.00 &1.00e-03 \\
\hline
\multirow{2}{*}{$c_{SD}$}&$\kappa_{DG}$&1.47e-05 &1.84e-06 &2.30e-07 &-&3.00 &3.00 &1.80e-03 &1.37e-06 &8.59e-08 &5.44e-09 &-&3.99 &3.98 &1.20e-03 \\
\cline{2-16}
&$\kappa_{+}$&1.47e-05 &1.84e-06 &2.30e-07 &-&3.00 &3.00 &1.80e-03 &1.37e-06 &8.59e-08 &5.44e-09 &-&3.99 &3.98 &1.20e-03 \\
\hline
\multirow{2}{*}{$c_{HU}$} &$\kappa_{DG}$&1.31e-05 &1.64e-06 &2.04e-07 &-&3.00 &3.00 &2.10e-03 &1.16e-06 &7.31e-08 &4.63e-09 &-&3.99 &3.98 &1.20e-03 \\
\cline{2-16}
&$\kappa_{+}$&1.31e-05 &1.64e-06 &2.04e-07 &-&3.00 &3.00 &2.10e-03 &1.16e-06 &7.31e-08 &4.63e-09 &-&3.99 &3.98 &1.20e-03 \\
\hline
\multirow{2}{*}{$c_{+}$}&$\kappa_{DG}$&8.28e-06 &1.02e-06 &1.28e-07 &-&3.01 &3.01 &2.30e-03 &8.39e-07 &5.25e-08 &3.32e-09 &-&4.00 &3.98 &1.20e-03 \\
\cline{2-16}
&$\kappa_{+}$&8.28e-06 &1.02e-06 &1.28e-07 &-&3.01 &3.01 &2.30e-03 &8.39e-07 &5.25e-08 &3.32e-09 &-&4.00 &3.98 &1.20e-03 \\
\hline
\end{tabular}
}
\caption{$L_{2}$-errors $p=3$.}
\label{tab:L2 Np4}
\end{table}

Tables \ref{tab:L2 Np3} and \ref{tab:L2 Np4} list the absolute $L_{2}$-errors for each grid, order of accuracy (OOA) and the maximum time step, $\Delta t_{max}$.  The order of convergence (OOA) is evaluated between 32/64 and 64/128 elements. The information from the tables can be summarized in the following two primary points. First, the relative error between $\Delta t_{max}$ obtained in these numerical experiments is less than 5.2\% against the results obtained from von Neumann analysis in Table~\ref{tab:Maximal time step}. Second, at even degree of polynomials for $c\ge c_{DG}$ and at both $\tau_{theory}$ and $1.5\tau_{theory}$ the OOA is $p+1$. At uneven degree of polynomials for all combinations of $c\ge c_{DG}$ and $\tau_{theory}$, one order of convergence is lost. This loss is recovered when $\tau$ is increased to $1.5\tau_{theory}$. Surprisingly, we observe that $c_{+}$ yields the least $L_{2}$-error. Further investigations have been conducted and the authors have observed that for $c\geq c_{+}$, the $L_{2}$-errors resume to increase. As expected, $\kappa$ has no influence. For the case $c_{DG}$, we observe that $\tau$ has an influence on the $L_{2}$-error but not on $\Delta t_{max}$.

\begin{figure}[H]
\centering

\begin{subfigure}{0.48\textwidth}
\centering
\includegraphics[width=3in]{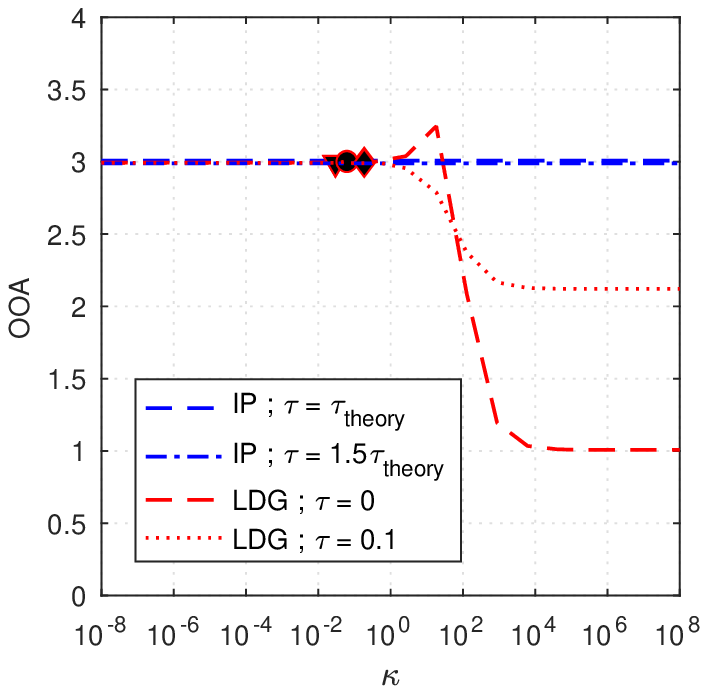}
\caption{Order of accuracy for $c=c_{+}$, $p=2$.} \label{fig:OOA DG p2}
\end{subfigure}\hspace*{\fill}
\begin{subfigure}{0.48\textwidth}
\centering
\includegraphics[width=3in]{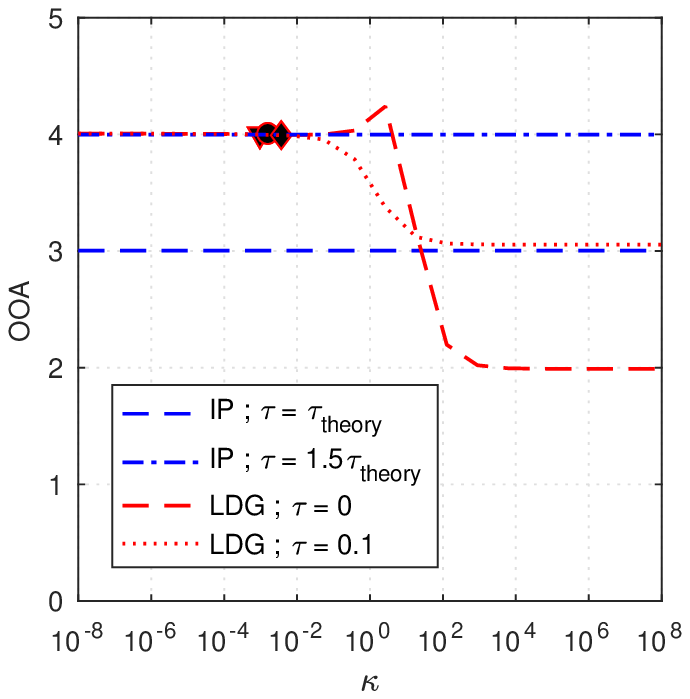}
\caption{Order of accuracy for $c=c_{DG}$, $p=3$.} \label{fig:OOA_p3_DG_IP_LDG_without_title.eps}
\end{subfigure}\hspace*{\fill}
\caption{OOA along $\kappa$ for different $c$ and $p$, the lower triangle is for $\kappa=\kappa_{SD}$, the circle is for $\kappa=\kappa_{HU}$ and the diamond for $\kappa=\kappa_{+}$; using log scale the DG case couldn't be represented.}
\label{fig:OOA}
\end{figure}

We provide a comparison of the order of accuacy between the IP and the LDG numerical fluxes. Orders are evaluated for $\kappa$ in the range $[10^{-8},10^{8}]$ and two values of $\tau$ at $\tau_{theory}$ and $1.5\tau_{theory}$. The OOA is calculated between the two meshes: 64 and 128 elements. For brevity, only the cases $c=c_{+}$; $p=2$ and $c=c_{DG}$; $p=3$ are presented in~Figure~\ref{fig:OOA}. Similar trends are observed for other values of $c$ and for higher $p$. Several observations can be drawn from~Figure~\ref{fig:OOA}. First, at even orders and for both values of $\tau$, the IP/BR2 schemes maintain a $p+1$ order. This is unlike LDG, where the orders are lost for $\kappa > \kappa_{+}$ for any $\tau$. Second, for uneven orders (~Figure~\ref{fig:OOA_p3_DG_IP_LDG_without_title.eps}), the IP/BR2 schemes lose an order for $\tau_{theory}$, but recover the lost order when $\tau$ is increased. Numerical simulations have demonstrated that $\tau=1.1 \tau_{theory}$ is sufficient to provide a constant expected value of OOA for any $p$; however, too large a value of $\tau$ diminishes $\Delta t_{max}$, as seen in Figure~\ref{fig:plot dt max}.

\begin{figure}[H]
\centering
\includegraphics[width=4in]{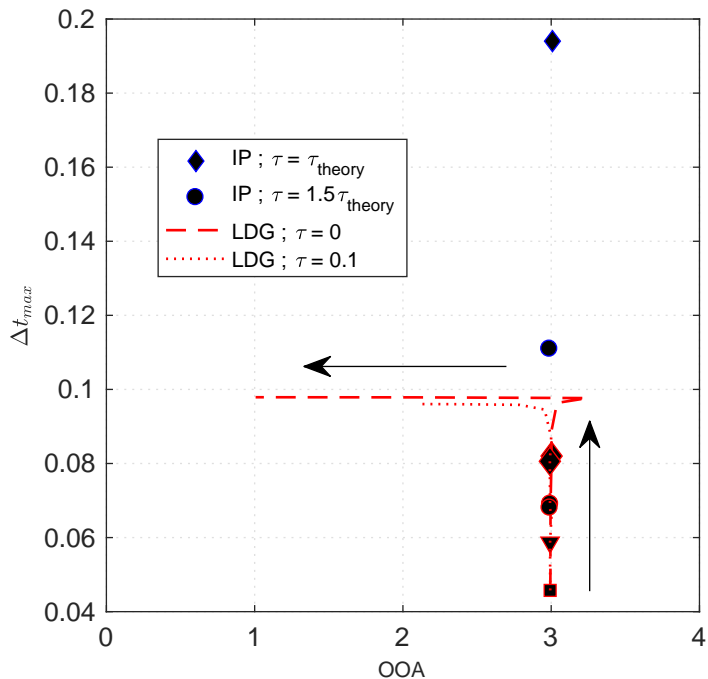}
\caption{Maximal time-step versus the order of accuracy for $p=2$ with $c=c_{+}$; the square is for $\kappa=\kappa_{DG}$, the lower triangle is for $\kappa=\kappa_{SD}$, the circle is for $\kappa=\kappa_{HU}$ and the diamond for $\kappa=\kappa_{+}$; the black arrows indicate the direction of increasing $\kappa$.}
\label{fig:dt VS OOA p2}
\end{figure}

\begin{figure}[H]
\centering
\includegraphics[width=4in]{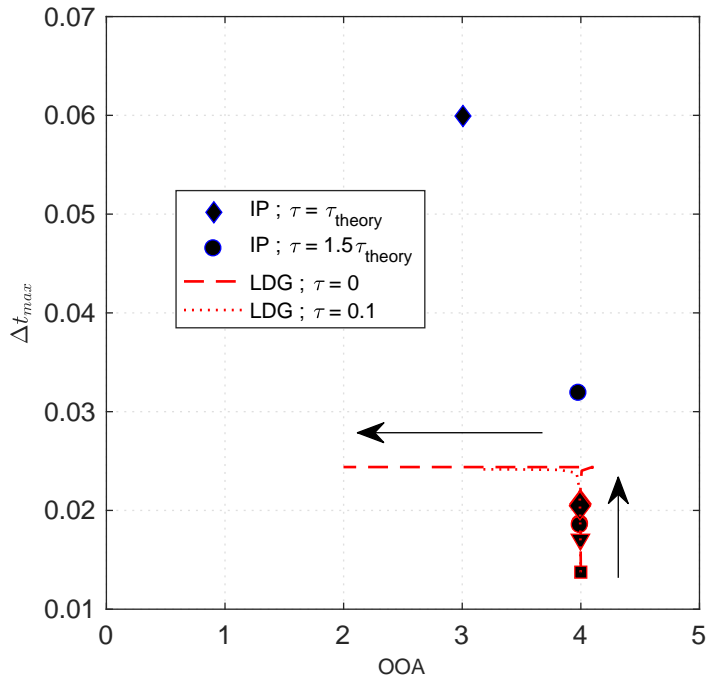}
\caption{Maximal time-step versus the order of accuracy for $p=3$ with $c=c_{+}$; the square is for $\kappa=\kappa_{DG}$, the lower triangle is for $\kappa=\kappa_{SD}$, the circle is for $\kappa=\kappa_{HU}$ and the diamond for $\kappa=\kappa_{+}$; the black arrows indicate the direction of increasing $\kappa$.}
\label{fig:dt VS OOA p3}
\end{figure}

One way to compare the schemes and better visualize the impact of $\tau$ on both the OOA and $\Delta t_{max}$ is through figures \ref{fig:dt VS OOA p2} and \ref{fig:dt VS OOA p3}.  We combine the previous figures and plot the maximal time step (obtained with the von Neumann analysis) versus the order of accuracy. One main conclusion can be summarized from these figures: for any order $p$, the IP/BR2 numerical fluxes provide a higher time step at the correct OOA than the LDG numerical fluxes.
\section{Conclusion}
This article provides a theoretical proof of energy stability for the diffusion case using the IP numerical fluxes. This proof was then extended to the BR2 numerical fluxes. For both these numerical fluxes, the problem does not depend on $\kappa$. Numerical simulations have shown that pairing $c$ close to $c_{+}$ with $\tau$ slightly greater than $\tau_{theory}$ provides both a high time step and a correct OOA. Further investigations will be conducted in higher dimensions.
\section*{Acknowledgements}
We would like to acknowledge the final support of Natural Sciences and Engineering Research Council of Canada Discovery Grant Program and McGill University.

\appendix

\section*{Appendix}
\section{Properties of Legendre polynomials}\label{sec:Properties Legendre polynomial}
Throughout this work, we use ESFR correction functions, defined by Legendre polynomials of degree $p$: $\Psi_{p}$.

Let us remind some properties of Legendre polynomial. The proof are omitted but can be found in any basic mathematical book.
\begin{enumerate}[label=Property \thesection .\arabic*,itemindent=*]
  \item Recurrence formula: Legendre polynomial follows Bonnet's recurrence formula\label{prop:Bonnet}
\begin{equation}\label{eq:Bonnet}
\left(p+1\right)\Psi_{p+1}\left(r\right)=\left(2p+1\right)\Psi_{p}\left(r\right)r-p\Psi_{p-1}\left(r\right).
\end{equation}
  \item \label{prop:symmetry}
  \begin{equation}\label{eq:symmetry}
  \Psi_{p}\left(-r\right)=\left(-1\right)^{p}\Psi_{p}\left(r\right).
  \end{equation}
  \item Differentiating equation~\eqref{eq:symmetry}\label{prop:derivative symetry}
  \begin{equation}\label{eq:derivative symetry}
  \Psi_{p}^{'}\left(r\right)=\left(-1\right)^{p+1}\Psi_{p}^{'}\left(-r\right).
  \end{equation}
  \item \label{prop:value +1 -1}
  \begin{equation}\label{eq:value +1 -1}
  \Psi_{p}\left(1\right)=1\,\,\text{and}\,\,\Psi_{p}\left(-1\right)=\left(-1\right)^{p}.
  \end{equation}
  \item Using~\ref{prop:Bonnet} and~\ref{prop:value +1 -1}, we can show by induction\label{prop:derivative value +1 -1}
\begin{equation}\label{eq:derivative value +1 -1}
\forall p\in\mathbb{N}, \Psi_{p}^{'}\left(1\right)=\frac{p\left(p+1\right)}{2} \,\,\text{and}\,\,\Psi_{p}^{'}\left(-1\right)=\left(-1\right)^{p+1}\frac{p\left(p+1\right)}{2}.
\end{equation}
\item \label{prop:integral legendre}
\begin{equation}\label{eq:integral legendre}
\displaystyle{
\int_{-1}^{1}\Psi_{i}\Psi_{j}\mbox{d}r=\dfrac{2}{2i+1}\delta_{ij},
}
\end{equation}
where $\delta_{ij}$ is the Kronecker delta.
\item\label{prop:derivative}It can be shown, using a generating function,
\begin{equation}
\Psi_{p+1}^{'}-\Psi_{p-1}^{'}=\left(2p+1\right)\Psi_{p}.
\end{equation}
\end{enumerate}
\section{ESFR correction functions}\label{sec:ESFR correction function}
With these properties, we obtain an interesting theorem for the ESFR correction functions.
\begin{definition}\label{def:Left}
A left-ESFR correction function, $g_{L}$ is, analytically, equal to: 
\begin{equation}\label{eq:left ESFR}
g_{L}=\dfrac{\left(-1\right)^{p}}{2}\left[\Psi_{p}-\left(\dfrac{\eta_{p,\kappa}\Psi_{p-1}+\Psi_{p+1}}{1+\eta_{p,\kappa}}\right)\right],
\end{equation}
where $\eta_{p,\kappa}=\frac{\kappa\left(2p+1\right)\left(a_{p}p!\right)^{2}}{2}$ and $a_{p}=\frac{\left(2p\right)!}{2^{p}\left(p!\right)^{2}}$.
\end{definition}
\begin{definition}\label{def:right}
A right-ESFR correction function, $g_{R}$ is, analytically, equal to
\begin{equation}\label{eq:right ESFR}
g_{R}=\dfrac{1}{2}\left[\Psi_{p}+\left(\dfrac{\eta_{p,\kappa}\Psi_{p-1}+\Psi_{p+1}}{1+\eta_{p,\kappa}}\right)\right].
\end{equation}
\end{definition}

\begin{thm}\label{thm:correction function}
Let $g_{L}$ be a left-ESFR correction function and $g_{R}$ be a right-ESFR correction function.
Then
\begin{eqnarray}
g_{L}^{'}\left(r\right)&=&-g_{R}^{'}\left(-r\right).\vspace{0.2cm}\label{eq:link correction function}
\end{eqnarray}
\end{thm}

\begin{proof}
Differentiating Definition~\ref{def:Left}, and then using equation~\eqref{eq:derivative symetry}
\begin{equation}
\begin{array}{lll}
g_{L}^{'}\left(r\right)&=&\dfrac{\left(-1\right)^{p}}{2}\left[\Psi_{p}^{'}\left(r\right)-\dfrac{\eta_{p,\kappa}\Psi_{p-1}^{'}\left(r\right)+\Psi_{p+1}^{'}\left(r\right)}{1+\eta_{p,\kappa}}\right]\vspace{0.2cm}\\
&=&\dfrac{\left(-1\right)^{p}}{2}\left[\left(-1\right)^{p+1}\Psi_{p}^{'}\left(-r\right)-\dfrac{\eta_{p,\kappa}\left(-1\right)^{p}\Psi_{p-1}^{'}\left(-r\right)+\left(-1\right)^{p+2}\Psi_{p+1}^{'}\left(-r\right)}{1+\eta_{p,\kappa}}\right]\vspace{0.2cm}\\
&=&\dfrac{\left(-1\right)^{2p+1}}{2}\left[\Psi_{p}^{'}\left(-r\right)+\dfrac{\eta_{p,\kappa}\Psi_{p-1}^{'}\left(-r\right)+\Psi_{p+1}^{'}\left(-r\right)}{1+\eta_{p,\kappa}}\right]\vspace{0.2cm}\\
g_{L}^{'}\left(r\right)&=&-g_{R}^{'}\left(-r\right).
\end{array}
\end{equation}
\end{proof}

\begin{rmk}
This Theorem is trivial if we consider the symmetry of the correction functions (\cite{vincent_new_2011}, Eq. (2.14)), $g_{L}\left(r\right)=g_{R}\left(-r\right)$.
\end{rmk}


\bibliography{bibliographie}

\end{document}